\documentclass[11pt,reqno]{amsart}

\usepackage{a4,graphicx,mathrsfs,amssymb, tikz, cancel, color}
\usepackage[colorlinks=true, citecolor=blue, linkcolor=red, urlcolor=black]{hyperref}
\usepackage{enumerate}
\usepackage{mathtools}
\usepackage{url}
\usepackage{pstricks-add}
\usepackage{caption}
\usepackage[backgroundcolor=yellow, colorinlistoftodos,prependcaption,textsize=small,textwidth=25mm]{todonotes}
\usepackage{cancel}
\usepackage[normalem]{ulem}

\newtheorem{theorem}{Theorem}[section]
\newtheorem{lemma}[theorem]{Lemma}
\newtheorem{corollary}[theorem]{Corollary}

\theoremstyle{definition}
\newtheorem{remark}[theorem]{Remark}
\newtheorem{definition}[theorem]{Definition}
\theoremstyle{remark}

\theoremstyle{definition}

\numberwithin{equation}{section}

\renewcommand{\AA}{\mathscr{A}}

\newcommand{\BB}{\mathscr{B}}

\newcommand{\HH}{\mathscr{H}}

\newcommand{\Ii}{\mathcal{I}}

\newcommand{\Pp}{\mathcal{P}}

\newcommand{\RR}{\mathcal{R}}

\newcommand{\Yy}{\mathscr{Y}}

\newcommand{\field}[1]{\mathbb{#1}}
\newcommand{\C}{\field{C}}
\newcommand{\R}{\field{R}}
\newcommand{\N}{\field{N}}
\newcommand{\Z}{\field{Z}}

\renewcommand{\Re}{\mathop{\text{\upshape{Re}}}}
\renewcommand{\Im}{\mathop{\text{\upshape{Im}}}}

\newcommand{\supp}{\mathop{\rm{supp}}}

\renewcommand{\div}{\mathop{\rm{div}}}

\newcommand{\ceqq}{\coloneqq}

\newcommand{\al}{\alpha}
\newcommand{\be}{\beta}

\newcommand{\de}{\delta}

\renewcommand{\th}{\theta}

\newcommand{\ka}{\kappa}
\newcommand{\la}{\lambda}
\newcommand{\rh}{\rho}
\newcommand{\sig}{\sigma}

\newcommand{\ph}{\varphi}

\newcommand{\ps}{\psi}

\newcommand{\Ga}{\Gamma}
\newcommand{\De}{\Delta}

\newcommand{\Om}{\Omega}

\newcommand{\na}{\nabla}

\newcommand{\pa}{\partial}

\renewcommand{\it}{\textit}

\renewcommand{\Re}{\mathop{\text{\upshape{Re}}}}
\renewcommand{\Im}{\mathop{\text{\upshape{Im}}}}

\renewcommand{\hat}{\widehat}
\renewcommand{\bar}[1]{\overline{#1}}
\renewcommand{\tilde}[1]{\widetilde{#1}}

\allowdisplaybreaks

\begin{document}
\title[Stability analysis for a plate-membrane system]{Stability analysis for a  plate-membrane system  without geometric condition}
\thanks{The first author thanks the DAAD for funding a research stay at the University of Konstanz.}

\author{Bienvenido Barraza Mart\'inez}
\author{Robert Denk}
\author{Jonathan Gonz\'alez Ospino}
\author{Jairo Hern\'andez Monz\'on}

\address{Departamento de Matem\'aticas y Estad\'istica - Universidad del Norte\vspace{-0.27 cm}}
\address{Km. 5 V\'ia Puerto Colombia, \'Area Metropolitana de Barranquilla, Colombia}
\email{bbarraza@uninorte.edu.co}
\email{gjonathan@uninorte.edu.co}
\email{jahernan@uninorte.edu.co}

\address{Fachbereich f\"ur Mathematik und Statistik - Universit\"at Konstanz\vspace{-0.27 cm}}
\address{Universit\"atsstrasse 10, 78464 Konstanz, Germany}
\email{robert.denk@uni-konstanz.de}

\renewcommand{\shortauthors}{B. Barraza Mart\'inez et al.}

\date{\today}



\begin{abstract}
In this work, we  consider a transmission problem describing a thermoelastic plate surrounding a membrane without any mechanical damping. The main results consist of the lack of exponential stability for this problem and the polynomial stability without the usual geometric condition. 
\end{abstract}

\maketitle

\section{Introduction}\label{intro}

In this work, we will investigate stability properties of a transmission problem where a thermoelastic plate in a region $\Omega_1$ is coupled with an elastic membrane in a region $\Omega_2$. More precisely, we consider the following geometric situation: Let $\Omega$ and $\Omega_2$ be bounded domains in $\mathbb{R}^2$ with enough regular (at least $C^4$) boundaries  $\Gamma\ceqq\partial\Omega$ and $I\ceqq\partial\Omega_2$, respectively, such that $\overline{\Omega_2}\subset\Omega$ and $\Omega_1\ceqq\Omega\smallsetminus\overline{\Omega_2}$. The unit outward normal vector to $\partial\Omega_1$ is denoted by $\nu\ceqq(\nu_1, \nu_2)^\top$ and $\tau\ceqq(-\nu_2, \nu_1)^\top$ is the unit tangent vector along $\partial\Omega_1$. Note that the unit outward normal vector to $I$ is $-\nu$. See Fig. \ref{grafica sin orificio}.

\begin{figure}[ht]
 \begin{center}
 \begin{tikzpicture}

 \pgfdeclareradialshading{thermal}{\pgfpoint{0cm}{0cm}}{
   color(0cm)=(blue!40);
   color(1cm)=(orange!25)
 }

 \shade[shading=thermal,shading angle=0,even odd rule]
   [rotate around={0.1746817133885854:(-3.7885809279481353,-4.01692271010987)}]
   (-3.7885809279481353,-4.01692271010987) ellipse (3.6654861587544563cm and 1.3521352683072208cm)
   [rotate around={2.254574965935038:(-2.5525173750121914,-4.349309547874158)}]
   (-2.5525173750121914,-4.349309547874158) ellipse (1.5013113362130832cm and 0.7148808572944548cm);

 \fill[gray!20]
   [rotate around={2.254574965935038:(-2.5525173750121914,-4.349309547874158)}]
   (-2.5525173750121914,-4.349309547874158) ellipse (1.5013113362130832cm and 0.7148808572944548cm);

 \draw [rotate around={0.1746817133885854:(-3.7885809279481353,-4.01692271010987)},line width=0.5pt]
   (-3.7885809279481353,-4.01692271010987) ellipse (3.6654861587544563cm and 1.3521352683072208cm);
 \draw [rotate around={2.254574965935038:(-2.5525173750121914,-4.349309547874158)},line width=0.5pt]
   (-2.5525173750121914,-4.349309547874158) ellipse (1.5013113362130832cm and 0.7148808572944548cm);

 \draw [->,line width=0.9pt] (-1.969665118853991,-3.6724395690215927) -- (-2.0256094708769915,-3.987446830768918);
 \draw [->,line width=0.9pt] (-1.5456816203920254,-2.9415587547002433) -- (-1.459805030539583,-2.6149953308742018);

 \draw[color=black] (-6.483542343039012,-4) node {$\Omega_1$};
 \draw[color=black] (-3.4,-4.4) node {$\Omega_2$};
 \draw[color=black] (-1.9,-4.15) node {$\nu$};
 \draw[color=black] (-1.5,-2.45) node {$\nu$};
 \draw[color=black] (-3,-3.4) node {$I$};
 \draw[color=black] (-5.9,-2.6) node {$\Gamma$};

 \end{tikzpicture}
 \end{center}
 \caption{The set $\Omega=\Omega_1\cup I\cup\Omega_2$.}\label{grafica sin orificio}
\end{figure}
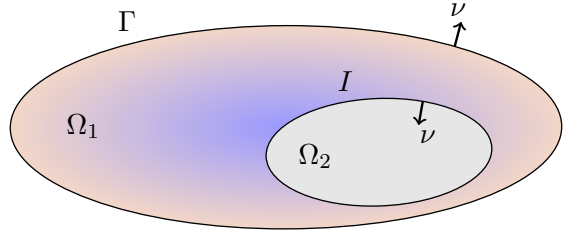

We will assume that $\Omega_1$ and $\Omega_2$ are occupied by the middle surface of the thermoelastic plate and by the membrane in equilibrium, respectively. In this case, $u=u(t, x)$ and $v=v(t, x)$ denote the vertical deflections of the plate and the membrane, respectively. The temperature difference of the plate, with respect to a reference temperature, is denoted by $\theta=\theta(t, x)$. We consider the following system of PDEs
\begin{align}
\rho_{1}u_{tt}+\beta_1\Delta^2 u+\alpha\Delta\theta&=0 \ \text{ in } \ \R^+\times\Omega_1,\label{ec1}\\
\rho_{0}\theta_{t}-\beta\Delta\theta +\sigma\theta -\alpha\Delta u_{t}&=0 \ \text{ in } \ \R^+\times\Omega_1,\label{ec2}\\
\rho_{2} v_{tt}-\beta_2\Delta v- m\Delta v_t &=0 \ \text{ in } \ \R^+\times\Omega_2.\label{ec3}
\end{align}
The constants $\alpha, \beta, \beta_1, \beta_2, \rho_0, \rho_1, \rho_2$ are all positive, while the constants $\sig$ and $m$ are considered non-negative. In \cite{bgh2023mana} there is a physical description of each of the above terms, \textcolor{black}{for example, $m$ represents the presence or absence of a Kelvin--Voigt type damping in the membrane.} In \cite{Barraza2021Long}, the authors considered the system \eqref{ec1}-\eqref{ec3} with and without frictional damping (also called viscous damping, see \cite{Gu2020Energy} and \cite{Zhao2005Stability}) on the membrane, $\sigma=m=0$, and transmission conditions on $I$ different from those introduced in this paper.

We will suppose that the plate is clamped on $\Gamma$ and also coupled to the membrane \textcolor{black}{on $I$}, \textcolor{black}{which} is interpreted by
\begin{equation}\label{condiciones naturales}
u=\partial_{\nu}u=0 \ \text{ on } \ \ \mathbb{R}^+\times\Gamma \hspace{0.5 cm}\text{ and }\hspace{0.5 cm} u=v \ \text{ on } \ \ \mathbb{R}^+\times I,
\end{equation}
respectively. We will assume that the temperature satisfies Newton's law of cooling (with coefficient $\kappa>0$, Robin case) along $\partial\Omega_1$, that is,
\begin{equation}\label{ley de enfriamiento}
\partial_{\nu}\theta+\kappa\theta=0 \ \text{ on } \
 \mathbb{R}^+\times\partial\Omega_1.
\end{equation}
We impose the following transmission conditions on the interface:
\begin{align}
\beta_1\mathscr{B}_1u+\alpha\theta&=0 \ \text{ on } \ \mathbb{R}^{+}\times I,\label{condtrans1}\\
\beta_1\mathscr{B}_2u +\alpha\partial_{\nu}\theta+\beta_2\partial_{\nu}v +m\partial_{\nu}v_t&=0 \ \text{ on } \ \mathbb{R}^{+}\times I,\label{ec13}
\end{align}
where
$$
\mathscr{B}_1u\ceqq \Delta u+(1-\mu)B_1u \ \text{ and } \ \mathscr{B}_2u \ceqq \partial_{\nu}\Delta u+(1-\mu)\partial_{\tau}B_2u,
$$
here $B_1$ and $B_2$ are the boundary operators defined by the relations
\begin{align*}
B_1u&\ceqq2\nu_1\nu_2u_{x_1x_2}-\nu_1^2u_{x_2x_2}-\nu_2^2u_{x_1x_1},\\
B_2u&\ceqq\nu_1\nu_2(u_{x_2x_2}-u_{x_1x_1})+(\nu_1^2-\nu_2^2)u_{x_1x_2}.
\end{align*}
The constant $\mu\in\left(0, \frac{1}{2}\right)$ is  Poisson ratio of the plate. Moreover, we consider the initial conditions
\begin{align}
\left(u, u_t, \theta\right)\big|_{t=0}&=(u_0, u_1, \theta_0) \ \text{ in } \ \Omega_1,\label{initial1}\\
\left(v, v_t\right)\big|_{t=0}&=(v_0, v_1) \ \text{ in } \ \Omega_2.\label{initial2}
\end{align}

Recently, in \cite{bgh2023mana} the authors proved  exponential stability of the semigroup associated with the system \eqref{ec1}-\eqref{initial2}, when $m>0$. But, what happens to the stability of the semigroup if $m=0$? Is it exponentially stable in this case?
Moreover, it was established in \cite{bgh2023mana} that  the semigroup associated with the system \eqref{ec1}-\eqref{initial2} decays polynomially of order at least 1/25, if the plate 
equation contains a rotational inertial term ($-\gamma\Delta u_{tt}$ with $\gamma >0$), \textcolor{black}{a structural damping acts on the  plate, } $m=0$, and if we assume the following geometrical condition: There exists a point $x_0\in\R^2$ such that
\begin{equation} \label{GC}
q(x)\cdot\nu(x)\leq0  \quad\textup{for}\quad x\in I,
\end{equation}
where the vector field $q$ is defined 
by $q(x)\ceqq x-x_0$. Then, it is interesting to ask now if this result holds when $\gamma=0$, $m=0$, and no mechanical damping acts on the thermoelastic plate. Another interesting question is the following: 
Can we obtain polynomial decay of the solution of \eqref{ec1}-\eqref{initial2} without geometric condition
on the boundary? 

It is common in the literature to assume this geometric condition in order to facilitate estimates for the $L^2$-norm of the wave component of the solution, which are necessary to obtain polynomial stability of the associated semigroup. In fact, using Rellich's integral identity, identities and inequalities can be obtained in which some integral terms involving the factor $q\cdot\nu$ can be discarded, allowing the desired estimates to be attained (see, e.g., \cite{Avalos2016Uniform,Barraza2021Long,BARRAZAMARTINEZ2019Regularity,bgh2023mana}). In our case, after using Rellich's identity, all the integral terms in equation \eqref{Eq_ast26b} are estimated and none are discarded. 

Transmission problems of materials composed of two different elastic components are related to the optimal design of material science, e.g. in damping mechanisms for bridges or in automotive industry  in other applications (see \cite{balmes2002tools, Germes2002, MunozGracia2007, Roy2006GermesBalmes} and the references therein). Hence, it is relevant to investigate  the asymptotic behavior of the solution of \eqref{ec1}-\eqref{initial2}, when $m\geq0$.
Specifically, we will prove in Section \ref{exponential instability} the
lack of exponential stability of the semigroup associated with system \eqref{ec1}-\eqref{initial2} if $m=0$, and in Section \ref{Polynomial stability} the polynomial stability of the semigroup associated if  $m=0$ and without taking into account the geometric condition \eqref{GC}.

Some words about notations: If $X$ is a Banach space, then its antidual is denoted by $X'$; for $\mathcal{O}\subset\R^2$ open, $S\subset\partial\mathcal{O}$, $s\in\R$, and $G\in\{\mathcal{O}, S\}$, $H^s(G)$ denotes the standard Sobolev spaces; the symbol $\hookrightarrow$ denotes continuous embedding, whereas $\overset{c}{\hookrightarrow}$ means that the continuous  embedding is compact. 

Throughout this document, the letter $C$ will stand for a generic constant which may vary in each time of appearance. Furthermore, the space of all bounded linear operators on a Banach space $X$ will be denoted by $\mathscr{L}(X)$.
Each space in this article is defined over the field of complex numbers.

\section{ Well-posedness, regularity, and spectral results}\label{Section2}

We first study the existence and uniqueness of problem \eqref{ec1}-\eqref{initial2}.  Let us consider the space $H_{\Gamma}^k(\Omega_1)\coloneqq\big\{w\in H^k(\Omega_1) : \frac{\partial^jw}{\partial\nu^j}=0 \ \text{on} \ \Gamma \ \text{for} \ j=0,\ldots, k-1\big\}$ with $k\in\N$. The space $H^2_{\Gamma}(\Omega_1)$ endowed with the scalar product
$$
\left(u, v\right)_{H^2_{\Gamma}(\Omega_1)}\coloneqq \mu\left(\Delta u, \Delta v\right)_{L^2(\Omega_1)}+(1-\mu)\left(\nabla^2u, \nabla^2v\right)_{L^2(\Omega_1)^4},
$$
where $u, v\in H^2_{\Gamma}(\Omega_1)$ and
$$
\left(\nabla^2u, \nabla^2v\right)_{L^2(\Omega_1)^4}\coloneqq \int_{\Omega_1}u_{x_1x_1}\bar v_{x_1x_1}+u_{x_2x_2}\bar v_{x_2x_2}+2u_{x_1x_2}\bar v_{x_1x_2} \ dx,
$$
is a Hilbert space, see \cite[Section 2]{BARRAZAMARTINEZ2019Regularity}. It is known that the norms $\|\cdot\|_{H^2_{\Ga}(\Om_1)}$ and $\|\cdot\|_{H^2(\Om_1)}$ are equivalent on $H^2_{\Ga}(\Om_1)$. We introduce the \emph{phase space}
$$
\HH\coloneqq \Big\{\ph=(\ph_1, \ph_2, \ph_3, \ph_4, \ph_5)^\top\in \mathscr X\,:\, \ph_1=\ph_3 \ \textup{on} \ I\Big\},
$$
where $\mathscr X\ceqq H^2_\Ga(\Om_1)\times L^2(\Om_1)\times H^1(\Om_2)\times L^2(\Om_2)\times L^2(\Om_1)$.  $\HH$ endowed with the inner product
\begin{align*}
\left(\ph, \ps\right)_{\HH}\coloneqq &\beta_1\left(\ph_1, \ps_1\right)_{H^2_{\Gamma}(\Omega_1)}+\rho_1\left(\ph_2, \ps_2\right)_{L^2(\Omega_1)}+\beta_2\left(\nabla\ph_3, \nabla\ps_3\right)_{L^2(\Omega_2)^2}\\
&+\rho_2\left(\ph_4, \ps_4\right)_{L^2(\Omega_2)}+\rho_0\left(\ph_5, \ps_5\right)_{L^2(\Omega_1)}
\end{align*}
for $\ph, \ps\in\HH$, is a Hilbert space. Note that the norm in $\HH$ induced by $(\cdot,\cdot)_{\HH}$ is equivalent to the usual norm in the product space $\mathscr{X}$.

If $w=(u, u_t, v, v_t, \theta)^{\top}$ is in an appropriate space, we can write \eqref{ec1}-\eqref{ec3} together with the initial conditions \eqref{initial1} and \eqref{initial2} as the following Cauchy problem
\begin{equation}\label{abstract with eta=0}
\partial_tw(t)=\AA w(t), \;t>0, \;\textup{ and }\; w(0)=w_0,
\end{equation}
where $w_0\ceqq(u_0, u_1, v_0, v_1, \theta_0)^{\top}$. The operator $\AA$ is defined by
{
\begin{equation}\label{Def_Aw}
\AA w\ceqq\begin{pmatrix}
w_2\\
-\frac{\beta_1}{\rh_1}\Delta^2 w_1 - \frac{\alpha}{\rh_1} \Delta w_5\\
w_4\\
\frac{\be_2}{\rh_2}\De w_3 \textcolor{black}{+ \frac{m}{\rho_2}\Delta w_4}\\
\frac{\al}{\rh_0} \Delta w_2 +\frac{\beta}{\rh_0} \Delta w_5 - \frac{\sigma}{\rh_0} w_5
\end{pmatrix},
\end{equation}}
with $w\coloneqq (w_1, w_2, w_3, w_4, w_5)^\top$. The domain of $\AA$ is defined such that $\AA w \in\HH$  for all $w\in D(\AA)$ \textcolor{black}{and that conditions \eqref{ley de enfriamiento}-\eqref{ec13} with $(u,v, v_t,\theta)=(w_1,w_3,w_4,w_5)$ are weakly satisfied,} that is,
\begin{align*}
D(\AA)\ceqq \big\{& w\in[H^2_\Ga(\Om_1)]^2\times[H^1(\Om_2)]^2\times L^2(\Om_1) : \Delta^{2}w_1, \De w_5\in L^2(\Om_1),\\
& \textcolor{black}{\beta_2\De w_3+m\Delta w_4\in L^2(\Om_2)},
\ w_j=w_{j+2} \ \textup{on} \ I \ \textup{for} \ j=1, 2 \\
& \text{ and } \text{\eqref{ley de enfriamiento}-\eqref{ec13}} 
  \text{ are weakly satisfied}\big\}.
\end{align*}
\noindent By straightforward calculation, we obtain the following equality.

\begin{lemma}\label{A is dissipative}
Let $m\geq0$. For $w=(w_1, w_2, w_3, w_4, w_5)^\top\in D(\AA)$, we have
\begin{align}\label{dissipativity}
\begin{split}
\Re\left(\AA w, w\right)_{\HH}=&-m\left\|\na w_4\right\|^2_{L^2(\Om_2)^2} -\sig\left\|w_5\right\|^2_{L^2(\Om_1)} -\be\left\|\na w_5\right\|^2_{L^2(\Om_1)^2}\\
&-\be\ka\left\|w_5\right\|^2_{L^2(\partial\Om_1)}.
\end{split}
\end{align}
Therefore, $\mathscr{A}$ is dissipative.
\end{lemma}
An application of the Lumer--Phillips theorem yields:

\begin{theorem}\label{C_0 semigroup}
For $ m \geq0$, the operator $\AA$ generates a $C_0$-semigroup $\left(\mathscr T(t)\right)_{t\geq0}$ of contractions on $\HH$.  Therefore, for any $w_0\in\HH$, the problem \eqref{abstract with eta=0} has a unique mild solution $w\in C([0, \infty), \HH)$. Furthermore, for any $w_0\in D(\AA^k)$ with $k\in\N$, there exists a unique classical solution $w$ to the problem \eqref{abstract with eta=0} that belongs to $\bigcap_{j=0}^kC^{k-j}([0, \infty), D(\AA^j))$.\\
\end{theorem}

By the above result, the resolvent set $\rho(\AA)$ contains the open half space $\C_+:= \{ \lambda\in\C: \Re\lambda >0\}$. 
It was shown in \cite[Proposition~4.4]{bgh2023mana} that for
$m>0$, we also have $i\R\subset\rho(\AA)$. Later, we will see that 
this  holds in the case $m=0$, too (see Lemma~\ref{Prop_resolvent}). 
In the calculations below, we will need the following fact, which also holds
for $m=0$. 

\begin{lemma}[{\cite[Proposition~4.1]{bgh2023mana}}]\label{Prop_IR}
If $m\geq0$, then $0\in\rh(\AA)$.
\end{lemma}

As the operator $\AA$ is defined in a weak sense, it is
not obvious that the functions taken in $D(\AA)$ gain enough regularity in such a way that the transmission conditions are satisfied in the strong sense of the trace. However, we have the following result from \cite[Remark~3.3]{bgh2023mana}.

\begin{theorem}\label{Th regularity}
Let $m\geq0$. If $w\in D(\AA)$, then $w_1\in H^4(\Om_1)$, $w_2\in H^2(\Om_1)$, $\beta_2 w_3 + m w_4\in H^2(\Om_2)$, $w_4\in H^1(\Om_2)$ and $w_5\in H^2(\Om_1)$. Therefore, if $w_0\in D(\AA)$ then $w(t)\coloneqq\mathscr T(t)w_0$ $(t\geq0)$ is the unique solution of the problem \eqref{ec1}-\eqref{initial2} and satisfies the boundary and transmission conditions in the strong sense of traces.
\end{theorem}

\begin{corollary}
\label{remarkregularity}
For $m=0$, we have 
\[
D(\mathscr{A})\hookrightarrow H^{4}\left(  \Omega_{1}\right)  \times
H^{2}\left(  \Omega_{1}\right)  \times H^{2}\left(  \Omega_{2}\right)  \times
H^{1}\left(  \Omega_{2}\right)  \times H^{2}\left(  \Omega_{1}\right).
\]
\end{corollary}

\begin{proof}
The fact that $D(\AA)$ is a subset of the product space on the right-hand side follows immediately from Theorem~\ref{Th regularity}. The embedding is continuous by the abstract argument from \cite[Lemma~3.5]{Barraza2021Long}.
\end{proof}

\begin{remark}
It was established in \cite{bgh2023mana} that the solutions of problem \eqref{ec1}-\eqref{initial2} have exponential stability when $m>0$, i.e., when the Kelvin--Voigt damping is active on the membrane. In this case, the semigroup  $\left(\mathscr T(t)\right)_{t\geq0}$ is exponentially stable in $\HH$, see \cite[Corollary 5.3]{bgh2023mana} and therefore, the solutions of problem \eqref{ec1}-\eqref{initial2} decays exponentially. The aim of the 
present paper is to show that without Kelvin--Voigt damping, we have no longer exponential stability, but still polynomial stability.
\end{remark}

\section{An a priori estimate}

Here, we will use the theory of parameter-elliptic boundary value problems to obtain an estimate relating the time derivative of the vertical displacement of the plate to the normal derivative of the vertical displacement of the membrane, which will be used in the next section. With respect to the  theory of parameter-ellipticity, we refer to \cite[Section 3]{DPRS23} and \cite[Section 2.5]{ThesisRau2023}.
To obtain the desired estimate, we  first consider  so-called Douglis--Nirenberg systems. We will restrict ourselves to linear differential operators with constant coefficients because this is sufficient for our purposes.

To begin with, let $N\in \N$, $\mathbf{s}:=(s_1,\dots,s_N), \mathbf{t}:=(t_1,\dots,t_N)\in \Z^N$ such that
$s_i\leq0$, $t_i\geq0$ for all $i\in\{1,\dots,N\}$, $s_1+t_1=\cdots=s_N+t_N=:\mu\in\N$ with $\mu N$ even. Let $A(D):=(A_{ij}(D))_{1\leq i, j\leq N}$ be an $N\times N$-matrix, where the entries $A_{ij}(D)$ are linear differential operators with constant coefficients and order $\mathrm{ord}(A_{ij}(D))\leq s_i+t_j$ if $s_i+t_j\geq0$ and $A_{ij}(D)=0$, whenever $s_i+t_j<0$. In this case, we say that the operator matrix $A(D)$ has a \emph{Douglis--Nirenberg structure}.

Now, let $M:=\frac{\mu N}{2}$, $\mathbf{m}:=(m_1,\dots,m_M)\in \Z^M$ with $m_k<0$ for all $k\in\{1,\dots,M\}$, and let $B(D):=(B_{kj}(D))_{1\leq k\leq M,\, 1\leq j\leq N}$ be an $M\times N$-matrix, whose entries are linear differential boundary operators with constant coefficients and order $\mathrm{ord}(B_{kj}(D))=m_k+t_j$ if $m_k+t_j\geq0$, and $B_{kj}(D)=0$, whenever $m_k+t_j<0$. 

Let $\mathcal{O}\subset \R^n$ be a bounded open set with boundary $\partial\mathcal{O}$ at least of class $C^{\max\{t_j\,:\,1\leq j\leq N\}}$. For $\lambda\in \Lambda$, where $\Lambda\subset\C$ is a closed sector in the complex plane with vertex at the origin, we will consider boundary value problems of the form
\begin{align}\label{Eq_DNsystem}
    \begin{split}
        (\lambda\mathbb{I}-A(D))w & = f \quad \text{in}\ \mathcal{O},\\
        B(D)w & = g \quad \text{on}\ \partial\mathcal{O},
    \end{split}
\end{align}
where the entries $I_{ij}$ of the $N\times N$-matrix $\mathbb{I}$ satisfy $I_{ij}=0$ for $i\neq j$, whereas $I_{ii}$ is the identity operator; the unknown  $w=(w_1,\dots,w_N)^T$ and the data  $f=(f_1,\dots,f_N)^\top$ belong to some suitable functions spaces defined in $\mathcal{O}$, and $g=(g_1,\dots,g_M)^\top$ belongs to some function space defined on $\partial\mathcal{O}$. For simplicity, we will write  $\lambda-A(D)$ instead $\lambda\mathbb{I}-A(D)$.  We say that the couple of operator matrices $(\lambda-A(D),B(D))$ has a \emph{Douglis--Nirenberg structure}, and  we call  \eqref{Eq_DNsystem} a \emph{Douglis--Nirenberg system}.

The entries of $A(D)$ and $B(D)$ are of the form
\begin{equation}\label{Eq_AD_BD}
    A_{ij}(D)=\!\!\!\!\sum\limits_{\alpha\in\N_0^n \atop \vert\alpha\vert\leq s_i+t_j}\!\!\! a_\alpha^{ij}D^\alpha \qquad \text{and} \qquad B_{kj}(D)=\!\!\!\!\sum\limits_{\beta\in\N_0^n \atop \vert\beta\vert\leq m_k+t_j}\!\!\! b_\alpha^{kj}\gamma_0D^\beta,
\end{equation}
respectively, where, as usual, $D:=-i\partial$ and $\gamma_0$ is the trace operator of order zero. The principal parts of    $A(D)$ and $B(D)$ are the $N\times N$ and $M\times N$ operators matrix $A^0(D)$ and $B^0(D)$ whose entries are defined analogously to  \eqref{Eq_AD_BD}, but considering only $\vert\alpha\vert=s_i+t_j$ and $\vert \beta\vert=m_k+t_j$, respectively. In the same way, the principal symbols of $A(D)$ and $B(D)$ are given by $A^0(\xi):=(A_{ij}^0(\xi))_{1\leq i,j\leq N}$ and $B^0(\xi):=(B_{kj}^0(\xi))_{1\leq k\leq M \atop 1\leq j\leq N}$, where
$$ A_{ij}^0(\xi)=\!\!\!\!\sum\limits_{\alpha\in\N_0^n \atop \vert\alpha\vert = s_i+t_j}\!\!\! a_\alpha^{ij}\xi^\alpha \qquad \text{and} \qquad B_{kj}^0(\xi)=\!\!\!\!\sum\limits_{\beta\in\N_0^n \atop \vert\beta\vert = m_k+t_j}\!\!\! b_\alpha^{kj}\xi^\beta, $$
respectively, for $\xi\in\R^n$.
\begin{definition}
    The operator matrix $\lambda - A(D)$ is called \emph{parameter-elliptic} in $\Lambda$  if 
    \begin{equation*}
        \det(\lambda - A^0(\xi))\neq 0 \quad (\lambda\in \Lambda,\  \xi\in \R^n,\  (\lambda,\xi)\neq (0,0)).
    \end{equation*}
    Here, $\lambda-A^0(\xi):=\lambda I_N - A^0(\xi)$, where $I_N$ is the identity matrix in $\C^{N\times N}$.\\
    The couple of operator matrices $(\lambda-A(D),B(D))$ associated to \eqref{Eq_DNsystem}, is called \emph{parameter-elliptic} in $\Lambda$ if $\lambda-A(D)$ is parameter-elliptic in $\Lambda$ and the following \emph{Lopatinskii--Shapiro condition} holds:

    Let $x_0$ be an arbitrary point of the boundary $\partial\mathcal{O}$, and rewrite $(\lambda-A(D),B(D))$ in the coordinate system associated to $x_0$, which is obtained from the original one by a rotation after which the positive $x_n$-axis has the direction of the interior normal to $\partial\mathcal{O}$ at $x_0$. Then, the trivial solution $\tilde{w}=0$ is the only stable solution of the ordinary differential equation on the half-line
    \begin{align*}
        (\lambda - A^0(\xi',D_n))\tilde{w}(x_n) & = 0 \quad (x_n>0),\\
        B^0(\xi',D_n)\tilde{w}(x_n)\vert_{x_n=0} & = 0,
    \end{align*}
    for $\xi'\in\R^{n-1}$ and $\lambda\in \Lambda$ with $(\xi',\lambda)\neq (0,0)$.
\end{definition}

Now, for $n=2$ and $T>0$, let us consider the thermoelastic plate
\begin{equation}\label{sistema3}
\begin{split}
\rho_1 u_{tt} + \beta_1 \Delta^2 u + \alpha \Delta \theta &= 0 \ \text{ in } \ (0,T) \times \mathcal{O},\\
\rho_0 \theta_t - \beta \Delta \theta + \sig\th-\al\Delta u_t &= 0 \ \text{ in } \ (0,T) \times \mathcal{O},
\end{split}
\end{equation}
with boundary conditions
\begin{equation*}\label{eq3_2}
u = 0, \quad \partial_\nu u = 0 \ \text{ and } \ \partial_\nu \theta +\ka\theta = 0 \ \text{ on } \ (0,T) \times \pa\mathcal{O},
\end{equation*}
or with the following boundary conditions
\begin{equation*}\label{eq3_3}
\beta_1\BB_1u +\alpha\theta = 0, \quad \beta_1\BB_2u + \alpha \partial_\nu \theta = 0 \ \text{ and } \ \partial_\nu \theta +\ka \theta = 0 \ \text{ on } \ (0,T) \times \pa\mathcal{O}.
\end{equation*}
\textcolor{black}{Here $\mathcal{O}\subset \R^2$ is a bounded open set with  boundary $\partial\mathcal{O}$ of class $C^4$. }

The PDE system \eqref{sistema3} can be written as the first-order system
$$
\partial_t U - AU = 0,
$$
where $ U = (u_1, u_2, u_3)^\top $ is associated to $ (u, u_t, \theta)^\top $ and $ A $ is given by
\begin{equation}\label{Def_A(D)}
A\ceqq A(D)\ceqq\begin{pmatrix}
0 & 1 & 0 \\
-\frac{\beta_1}{\rho_1} \Delta^2 & 0 & -\frac{\alpha}{\rho_1} \Delta \\
0 & \frac{\al}{\rho_0} \Delta & \frac{\beta}{\rho_0} \Delta - \frac{\sigma}{\rho_0}
\end{pmatrix}
\end{equation}
with principal part
$$
A^0(D)\ceqq\begin{pmatrix}
0 & 1 & \phantom{0}0 \\
-\frac{\beta_1}{\rho_1} \Delta^2 & 0 & -\frac{\alpha}{\rho_1} \Delta \\
0 & \frac{\al}{\rho_0} \Delta & \phantom{-}\frac{\beta}{\rho_0} \Delta
\end{pmatrix}.
$$
The operator $ A $ have  a Douglis--Nirenberg structure  with $ \mathbf{s} = (-2,0,0) $ and $ \mathbf{t} = (4,2,2)$. We will analyze the parameter-ellipticity for the couples of operator matrices $ (\lambda - A, B_1) $ and $ (\lambda - A, B_2)$, where the boundary operator matrices on $ \partial \mathcal{O} $ are given by

\begin{equation}\label{Def_B12(D)}
{\small B_1\ceqq B_1(D)\ceqq \begin{pmatrix}
1 & 0 & 0 \\
\partial_\nu & 0 & 0 \\
0 & 0 & \partial_\nu + \ka
\end{pmatrix}\ \text{ and } \
B_2\ceqq B_2(D)\ceqq\begin{pmatrix}
\beta_1 \BB_1 & 0 & \alpha \\
\beta_1 \BB_2 & 0 & \alpha \partial_\nu \\
0 & 0 & \partial_\nu + \ka
\end{pmatrix}.}
\end{equation}
Now, the boundary operators $ \BB_1 $ and $ \BB_2 $ can be expressed in terms of normal and tangential derivatives (see \cite{LasieckaTriggiani2000Control}, Propositions 3C.7 and 3C.11),
\begin{align*}
\BB_1 u&=\partial_\nu^2 u + \mu \partial_\tau^2 u + \mu (\text{div}\nu)\partial_\nu u,\\
\BB_2 u&=\partial_\nu^3 u + \partial_\nu \partial_\tau^2 u + (1 - \mu)\pa_\tau\pa_\nu\pa_\tau u + \partial_\nu[(\text{div}\nu) \partial_\nu u].
\end{align*}
After a rotation such that $ -\nu $ is in the direction of the  positive {$x_2$-}axis and $ \tau $ is in the direction of the positive {$x_1$-}axis, we obtain that the principal part of $ B_1 $ and $ B_2 $ are
$$
B_1^0(D)\ceqq\begin{pmatrix}
1 & 0 & 0 \\
-iD_2 & 0 & 0 \\
0 & 0 & -iD_2
\end{pmatrix}
$$
and
$$
B_2^0(D)\ceqq\begin{pmatrix}
-\beta_1(D_2^2 + \mu D_1^2) & 0 & \alpha \\
i\beta_1[D_2^3 + D_2D_1^2 + (1 - \mu) D_1D_2D_1] & 0 & -i\al D_2 \\
0 & 0 & -iD_2
\end{pmatrix},
$$
respectively. We recall that  $D_j = -i\partial_j$. Note that the couples of operators matrices $ (\lambda - A, B_1) $ and $ (\lambda - A, B_2) $ have Douglis--Nirenberg structures with $\mathbf{m}_1 = (-4, -3, -1)$ and $\mathbf{m}_2 = (-2, -1, -1)$.

\begin{lemma}\label{lemma1}
The couples of operators matrices  $ (\lambda - A, B_1) $ and $ (\lambda - A, B_2) $ are parameter-elliptic in the closed sector $ \overline{\mathbb{C}_+} $.
\end{lemma}
\begin{proof}
First, note that
\begin{align*}
\det(\lambda - A^0(\xi)) &= \det \begin{pmatrix}
\la & -1 & 0 \\
\frac{\beta_1}{\rho_1} |\xi|^4 & \lambda & -\frac{\alpha}{\rho_1} |\xi|^2 \\
0 & \frac{\al}{\rho_0} |\xi|^2 & \lambda + \frac{\beta}{\rho_0} |\xi|^2
\end{pmatrix}\\
&= \lambda \Big(\la\big(\lambda + \tfrac{\beta}{\rho_0} |\xi|^2\big) + \tfrac{\alpha^2}{\rho_0 \rho_1} |\xi|^4\Big) - \tfrac{\beta_1}{\rho_1}\vert \xi\vert^4 \big( -(\lambda + \tfrac{\beta}{\rho_0} |\xi|^2)-0\big)\\
&= \lambda^3 + \tfrac{\beta}{\rho_0} |\xi|^2 \lambda^2 + \tfrac{\alpha^2+\be_1\rh_0}{\rho_0 \rho_1} |\xi|^4 \lambda + \tfrac{\be\beta_1}{\rho_0 \rho_1} |\xi|^6
\end{align*}
for $ \xi \in \mathbb{R}^n $ and $ \lambda \in \overline{\mathbb{C}_+} $. It is clear that this determinant is different from zero if $ \xi=0 $ and $ \lambda \neq 0 $. If $ \xi \in \mathbb{R}^n\smallsetminus\{0\}$, then
$$
\det(\lambda - A^0(\xi)) = |\xi|^6 p\left( \frac{\lambda}{|\xi|^2} \right)
$$
with
$$
p(\lambda)\ceqq\lambda^3 + a \lambda^2 + b \lambda + c, \quad a\ceqq \frac{\beta}{\rho_0}, \quad b\ceqq\frac{\alpha^2 + \beta_1 \rho_0}{\rho_0 \rho_1} \ \text{ and } \ c\ceqq\frac{\be\beta_1}{\rho_0 \rho_1}.
$$
Then, from Lemma \ref{lemmaAN}, it follows that $ \lambda - A(D) $ is parameter-elliptic in $ \overline{\mathbb{C}_+} $.

We only prove the Lopatinskii--Shapiro condition for  $ (\lambda - A, B_2) $. The proof of this condition for the couple $ (\lambda - A, B_1) $ is similar and easier. For the proof for $ (\lambda - A, B_2) $, we consider the system of ordinary differential equations with initial conditions
\begin{align*}
\lambda - A^0(\xi_1, D_2) \widetilde{w}(x_2) &= 0 \quad (x_2 > 0),\\
B_2^0(\xi_1, D_2) \widetilde{w}(0) &= 0,
\end{align*}
for all $ (\xi_1, \lambda) \in \mathbb{R} \times \overline{\mathbb{C}_+} $ with $ (\xi_1, \lambda) \neq (0,0) $. Let $ \widetilde{w} { =: } (u, \eta, \theta)^\top $ be a stable solution of this problem.
Then, $ \eta = \lambda u $, and therefore
\begin{align}
\rho_1 \lambda^2 u + \beta_1 (\partial_2^2 - \xi_1^2)^2 u + \alpha (\partial_2^2 - \xi_1^2) \theta &= 0,\label{eq3_4}\\
\alpha\la(\xi_1^2 - \partial_2^2) u + \rho_0 \lambda \theta + \beta (\xi_1^2 - \partial_2^2) \theta &= 0,\label{eq3_5}
\end{align}
with  initial conditions
\begin{equation}\label{eq3_6}
\begin{split}
\beta_1 ((\partial_2^2 - \mu \xi_1^2) u)(0) + \alpha \theta(0) &= 0,\\
((\partial_2^3 - (2 - \mu) \xi_1^2\partial_2) u)(0) &= 0, \\
\theta'(0) &= 0.
\end{split}
\end{equation}
Multiplying \eqref{eq3_4} by $ \overline{\lambda u} $, \eqref{eq3_5} by $ \overline{\theta} $, summing up and integrating over $ (0, \infty) $, we obtain
\begin{equation}\label{eq3_7}
\begin{split}
0 &= \rho_1\la|\lambda|^2 \| u \|_{L^2(0,\infty)}^2 + \beta_1 \overline\lambda((\partial_2^2 - \xi_1^2)^2 u, u)_{L^2(0,\infty)}\\
&\quad + \alpha\overline\la((\partial_2^2 - \xi_1^2) \theta, u)_{L^2(0,\infty)} + \alpha\lambda((\xi_1^2 - \partial_2^2) u, \theta)_{L^2(0,\infty)} \\
&\quad + \rh_0\|\th\|^2_{L^2(0,\infty)}+\be((\xi_1^2 - \partial_2^2) \theta, \theta)_{L^2(0,\infty)}.
\end{split}
\end{equation}
Now, we know from (4.2) in \cite{BARRAZAMARTINEZ2019Regularity} that
$$
((\partial_2^2 - \xi_1^2) u, u)_{L^2(0,\infty)} =\mathcal{P}+[(\partial_2^2 - \mu \xi_1^2) u](0)\partial_2\overline{u(0)} + [(-\partial_2^3+(2 - \mu) \xi_1^2 \partial_2)u](0)\overline{u(0)},
$$
where
$$
\Pp\ceqq\mu \| (\partial_2^2 - \xi_1^2) u \|_{L^2(0,\infty)}^2 + (1 - \mu)\big( \| \xi_1^2 u \|_{L^2(0,\infty)}^2 + \| \partial_2^2 u \|_{L^2(0,\infty)}^2
+ 2 \| \xi_1 \partial_2 u \|_{L^2(0,\infty)}^2\big).
$$
From this, \eqref{eq3_6}, and \eqref{eq3_7}, it follows
\begin{align*}
0 &= \rho_1\la|\lambda|^2 \| u \|_{L^2(0,\infty)}^2 + \beta_1\overline{\lambda} \mathcal{P}
+ \alpha \overline{\lambda} \Big(-\xi_1^2(\theta, u)_{L^2(0,\infty)}
+ \int_0^\infty \partial_2^2 \theta \, \overline{u} \, dx_2 \Big)\\
&\quad + \alpha \lambda\Big( \xi_1^2(u, \theta)_{L^2(0,\infty)}
- \int_0^\infty \partial_2^2 u \, \overline{\theta} \, dx_2\Big)
+ (\rho_0 \lambda + \beta \xi_1^2) \| \theta \|^2_{L^2(0,\infty)}\\
&\quad-\beta\int_0^\infty \partial_2^2 \theta \, \overline{\theta} \, dx_2.
\end{align*}
After  integration by parts, we obtain
\begin{align*}
0&=\rho_1\la|\lambda|^2 \| u \|_{L^2(0,\infty)}^2+\beta_1\overline\la \mathcal{P}+i2\al\xi_1^2\Im((u, \th)_{L^2(0, \infty)}+\la(\pa_2u, \pa_2\th)_{L^2(0, \infty)})\nonumber\\
&\quad+(\rh_0\la+\be\xi_1^2)\|\th\|^2_{L^2(0, \infty)}+\be\|{\partial_2\th}\|^2_{L^2(0, \infty)}
\end{align*}
and therefore
$$
0 = (\Re\lambda)\big[\rho_1| \lambda|^2 \| u \|_{L^2(0,\infty)}^2+\beta_1\Pp\big] + (\rho_0 \Re \lambda + \beta \xi_1^2) \| \theta \|_{L^2(0,\infty)}^2
+ \beta\| {\partial_2\th} \|_{L^2(0,\infty)}^2.
$$
Because of $ \Re\lambda\geq0 $ {and $\Pp\ge 0$}, it follows that $ {\partial_2\th}=0 $. Since {$\th(x_2)\to 0$ for $x_2\to\infty$, we get} $ \theta = 0 $. Thus,
$$
0 = (\Re\lambda) \big[ \rho_1|\la|^2 \|u\|_{L^2(0,\infty)}^2 + \beta_1\mathcal{P}\big].
$$
If $ \lambda \neq 0 $, it follows from $ \theta = 0 $ and equation \eqref{eq3_5} that
$$
(\xi_1^2 - \partial_2^2) u = 0 \ \text{ in } \ (0, \infty).
$$
From this and \eqref{eq3_4}, we obtain
$$
0 = \rho_1 \lambda^2 u + \beta_1 (\partial_2^2 - \xi_1^2)((\partial_2^2 - \xi_1^2)u) + \alpha (\partial_2^2 - \xi_1^2) \theta
= \rho_1 \lambda^2 u,
$$
and so, $ u = 0 $. If $ \lambda = 0 $, it follows from \eqref{eq3_4} that $(\partial_2^2 - \xi_1^2)^2 u = 0$, with $(\partial_2^2 - \mu \xi_1^2) u(0) = 0$ and $[(-\pa^3_2+(2-\mu)\xi^2_1\pa_2)u](0)=0$. Then $ u = 0 $ because the negative of the biharmonic operator, supplemented with the free boundary operators, is parameter-elliptic (see Lemma 4.1 in \cite{BARRAZAMARTINEZ2019Regularity}).
\end{proof}

\textcolor{black}{For $\mathbf{r} \ceqq (r_1,r_2,r_3)\in\R^3$ let $H^{\mathbf{r}}(\mathcal{O})$ be the product of Sobolev spaces given by $H^{\mathbf{r}}(\mathcal{O})\ceqq H^{r_1}(\mathcal{O})\times H^{r_2}(\mathcal{O}) \times H^{r_3}(\mathcal{O})$ with the usual norm. Moreover, for $q\in\R$ let $\mathbf{r}+q \ceqq (r_1+q,r_2+q,r_3+q)$.}

\begin{lemma}\label{Prop 3.2}
Let now $ \mathcal{O} \subset \mathbb{R}^2 $ be a domain with $ C^4 $-boundary
$ \partial \mathcal{O} {=}\Gamma_1 \cup \Gamma_2 $,
where $ \Gamma_1 $ and $ \Gamma_2 $ are compact, disjoint parts of the boundary of $ \mathcal{O} $\textcolor{black}{, $\mathbf{s}\ceqq (-2,0,0)$, and $\mathbf{t}\ceqq(4,2,2)$.}
We define the operator $ \mathcal{A} : \mathbb{H} \supset D(\mathcal{A}) \rightarrow \mathbb{H} $ by
$$
\mathbb{H} \ceqq \{ w=(w_1,w_2,w_3) \in H^{-\mathbf{s}}(\mathcal{O}): w_1 = \partial_\nu w_1 = 0 \textup{ on } \Gamma_1 \},
$$
$$
D(\mathcal{A}) \ceqq \left\{ w \in H^{\mathbf{t}}(\mathcal{O}): \mathcal{A}(D)w \in \mathbb{H},\; B_1(D)w = 0 \textup{ on } \Gamma_1 \ \textup{and} \ B_2(D)w = 0 \textup{ on } \Gamma_2 \right\}
$$
and $ \mathcal{A} w \ceqq A(D)w $, where $A(D)$, $B_1(D)$ and $B_2(D)$ \textcolor{black}{are given in  \eqref{Def_A(D)} and \eqref{Def_B12(D)}}.  On $ \mathbb{H} $ we consider the inner product induced by
$$
\| w \|_{\mathbb{H}}^2 \ceqq \| w_1 \|_{H^2_{\Ga_1}(\mathcal{O})}^2 + \| w_2 \|_{L^2(\mathcal{O})}^2 + \| w_3 \|_{L^2(\mathcal{O})}^2.
$$
Then $ \mathcal{A} $ is the generator of a $ C_0 $-semigroup of contractions on $ \mathbb{H} $.
\end{lemma}
\begin{proof}
Note that $\mathbb{H} = H^2_{\Ga_1}(\mathcal{O}) \times L^2(\mathcal{O}) \times L^2(\mathcal{O})$, and also
\begin{align*}
D(\mathcal{A})& = \{ w \in \big(H^4(\mathcal{O}) \cap H^2_{\Ga_1}(\mathcal{O})\big) \times H^2_{\Ga_1}(\mathcal{O}) \times H^2(\mathcal{O}):
\beta_1\BB_1 w_1 + \alpha w_3 = 0 \textup{ on } \Gamma_2, \\
&\qquad\beta_1\BB_2 w_1 + \alpha \partial_\nu w_3 = 0 \textup{ on } \Gamma_2 \ \text{ and } \ \partial_\nu w_3 + \ka w_3 = 0 \textup{ on } \partial \mathcal{O}\}.
\end{align*}
The norm in $ \mathbb{H} $ is equivalent to the standard norm in
$ H^2(\mathcal{O}) \times L^2(\mathcal{O}) \times L^2(\mathcal{O}) $. By integration by parts, we obtain that
$$
\Re(\mathcal{A} w, w)_{\mathbb{H}} \leq 0 \quad \forall w \in D(\mathcal{A}),
$$
and therefore $ \mathcal{A} $ is dissipative.

 By Remark 4.4 in \cite{BARRAZAMARTINEZ2019Regularity} and Lemma \ref{lemma1}, we have that the problem
\begin{align}\label{Problem_parameter-elliptic-AB}
\begin{split}
(\lambda - A(D))w &= (0,0,0)^\top \ \text{ in } \ \mathcal{O}, \\
B_1(D)w &= (0,0,0)^\top \ \text{ on } \ \Gamma_1, \\
B_2(D)w &= (0,0,0)^\top \ \text{ on } \ \Gamma_2,
\end{split}
\end{align}
is parameter-elliptic in $ \overline{\mathbb{C}_+} $.
Because of the coefficient functions of the operators $ A(D), B_1(D) $ and $ B_2(D) $ are constant, one can show similarly to the scalar case in Theorem 4.9 in \cite{DPRS23} that there is a   $ \lambda_0 > 0 $ such that for every $f\in\mathbb H$ there exists $ w \in D(\mathcal{A}) $ such that
$$
(\lambda_0 - \mathcal{A}) w = f.
$$
This shows that $ \lambda_0 - \mathcal{A} $ is surjective.
Now, by the Lumer--Phillips theorem, we conclude that $ \mathcal{A} $ generates a $ C_0 $-semigroup of contractions on $ \mathbb{H} $.
\end{proof}

 The following result is the desired a priori estimate. In the proof, we use results on parameter-elliptic Douglis--Nirenberg systems including estimates in parameter-dependent norms. For this, we define for $s\in\R$ and $u\in H^s(\R^n)$ the parameter-dependent norm 
\[ \|u\|_{H^s_\lambda(\R^n)}  := \langle\lambda\rangle^{s/2-n/4} \Big\| u\Big(\frac{\cdot}{\langle \lambda\rangle^{1/2}}\Big)\Big\|_{H^s(\R^n)}. \]
The norms $\|u\|_{H^s_\lambda(\R^{n-1})}$ and $\|u\|_{H^s_\lambda (\mathcal O)}$ are defined analogously and by restriction, respectively, while the norm on the boundary $\partial\mathcal O$ is defined by localization and a partition of unity. As usual, we define $H^{\mathbf{t}}_\lambda(\mathcal O) := H^{t_1}_\lambda(\mathcal O)\times H^{t_2}_\lambda(\mathcal O)\times H^{t_3}_\lambda(\mathcal O) $ for $\mathbf{t}=(t_1,t_2,t_3)$. 

For a discussion of these parameter-dependent norms, we refer to \cite[Subsection~1.1]{Grubb-Kokholm93}. For $s>0$, we have the equivalence of the norms
\[ \|u\|_{H^s_\lambda(\R^n)} \approx \|u\|_{H^s(\R^n)} + |\lambda|^{s/2} \|u\|_{L^2(\R^n)}. \]
It was shown in \cite[Theorem~3.3.5]{ThesisRau2023} that, under appropriate smoothness assumptions, parameter-elliptic Douglis--Nirenberg systems are uniquely solvable for sufficiently large $\lambda$ and that their solutions satisfy a uniform a priori estimate with respect to these parameter-dependent norms (see also \cite[Theorem 1.5.1]{Roitberg99}; for scalar equations this result can be found in \cite[Theorem~4.5]{DPRS23}).

\begin{lemma}\label{Prop 3.3}
Suppose $\mathcal{O}$ possesses now a $C^5$-boundary $\partial\mathcal{O}=\Gamma_1\cup\Gamma_2$ with $\Gamma_1$ and $\Gamma_2$ as above. Let $A$, $B_1$, $B_2$, $\mathbf{s}$, and $\mathbf{t}$ be as before. Furthermore, let $T>0$ and
$$ g\in L^2((0,T),H^{1/2}(\Gamma_2))\cap H^{1/4}((0,T),L^2(\Gamma_2)). $$
Let $w\in C([0,T],H^{\mathbf{t}}(\mathcal{O})\cap C^1([0,T],H^{-\mathbf{s}}(\mathcal{O}))$ be the unique solution of 
\begin{align}\label{Problem_a-priori}
\begin{split}
\partial_t w - A(D)w & = (0,0,0)^\top \qquad \text{in}\ (0,T)\times \mathcal{O},\\
B_1(D)w & = (0,0,0)^\top \qquad \text{on}\ (0,T)\times \Gamma_1,\\
B_2(D)w & = (0,g,0)^\top \qquad \text{on}\ (0,T)\times \Gamma_2,\\
w(0,\cdot) & \ = (0,0,0)^\top \qquad \text{in}\ \mathcal{O},
\end{split}
\end{align}
Then it holds
$$ \Vert w\Vert_{L^2((0,T) , H^{\mathbf{t}-1/2}(\mathcal{O}))\cap H^{(2\mathbf{t}-1)/4}((0,T),L^2(\mathcal{O}))} \leq C\Vert g\Vert_{L^2((0,T),L^2(\Gamma_2))}. $$
\end{lemma}

\begin{proof}
In a first step,  we  extend the function  $g$  appearing in problem \eqref{Problem_a-priori} by zero   to a function  $\tilde{g}\in L^2((-\infty,\infty),H^{1/2}(\Gamma_2))\cap H^{1/4}((-\infty,\infty),L^2(\Gamma_2))$ with $\supp \tilde{g}\subset [0,T]$ (see \cite[Theorem 1, Section 4.3.2.]{Triebel78}). 
Define $\tilde w$ as the unique solution of \eqref{Problem_a-priori} in the time interval $(-\infty , \infty)$ with $g$ being replaced by $\tilde g$. Then, by uniqueness of the solution on finite time intervals, we see that $\tilde w$ is an extension of $w$. 
Now, we consider the problem \eqref{Problem_a-priori} in $(-\infty,\infty)$ instead of $(0,T)$ and with $\tilde{\phi}$ instead of $\phi$ for $\phi\in\{w, g\}$. 

In the next step, we apply the one-dimensional Fourier transform $\mathscr{F}_{t\to\lambda}$ to obtain a parameter-dependent problem. As we want to apply the results on parameter-elliptic Douglis--Nirenberg systems to all $\lambda\in\R$, we shift the parameter $i\lambda$ to $\lambda_1+i\lambda$ with 
sufficiently large $\lambda_1>0$. For this, we multiply the new system by $e^{-\lambda_1t}$. Denoting $\hat{\phi}(\lambda):=\mathscr{F}_{t\to \lambda}(e^{-\lambda_1(\cdot)}\tilde{\phi})(\lambda)$ for $\phi\in\{w, g\}$,  we obtain 
\begin{align}\label{Problem_a-priori_lambda}
\begin{split}
(\lambda_1+i\lambda - A(D))\hat{w}(\lambda) & = (0,0,0)^\top \qquad \  \text{in}\ \mathcal{O},\\
B_1(D)\hat{w}(\lambda) & = (0,0,0)^\top \qquad \ \text{on}\  \Gamma_1,\\
B_2(D)\hat{w}(\lambda) & = (0,\hat{g}(\lambda),0)^\top \quad \text{on}\  \Gamma_2,
\end{split}
\end{align}
for $\lambda \in\R$. Due to Lemma \ref{lemma1}, we have that $ (\lambda_1 + i\lambda - A, B_1) $ and $ (\lambda_1+i\lambda - A, B_2) $ are parameter-elliptic in the parameter-range $\lambda_1+i\lambda\in\overline{\C_+}$. Therefore, we can apply \cite[Theorem~3.3.5]{ThesisRau2023} which yields that for sufficiently large $\lambda_1>0$, the boundary value problem \eqref{Problem_a-priori_lambda} has a unique solution 
$\hat{w}(\lambda)\in  H^{\mathbf{t}}_{\lambda_1+i\lambda}(\mathcal{O})$, for which the a priori estimate
\begin{equation*}
    \Vert \hat{w}(\lambda)\Vert_{H^{\mathbf{t}-1/2}_\lambda(\mathcal{O})} \leq C \Vert \hat{w}(\lambda)\Vert_{H^{\mathbf{t}-1/2}_{\lambda_1+i\lambda}(\mathcal{O})} \leq C \Vert \hat{g}(\lambda)\Vert_{L^2(\Gamma_2)}
\end{equation*}
holds for almost every $\lambda\in\R$. 

As the Fourier transform is an isometric isomorphism, one can show that 
\[  \Vert \widetilde{w} \Vert_{L^2(\R , H^{\mathbf{t}-1/2}(\mathcal{O}))\cap H^{(2\mathbf{t}-1)/4}(\R,L^2(\mathcal{O}))}^2\approx \int_\R \Vert \hat{w}(\lambda)\Vert^2_{H^{\mathbf{t}-1/2}_\lambda (\mathcal{O})}\,d\lambda \]
(see also \cite[Lemma~3.4.1]{ThesisRau2023}). Therefore, we can estimate 
\begin{align*}
    \Vert w &\Vert_{L^2((0,T) , H^{\mathbf{t}-1/2}(\mathcal{O}))\cap H^{(2\mathbf{t}-1)/4}((0,T),L^2(\mathcal{O}))}^2\\
    &  \le  \Vert \widetilde{w} \Vert_{L^2(\R , H^{\mathbf{t}-1/2}(\mathcal{O}))\cap H^{(2\mathbf{t}-1)/4}(\R,L^2(\mathcal{O}))}^2\\
    & \leq C \int_\R \Vert \hat{w}(\lambda)\Vert^2_{H^{\mathbf{t}-1/2}_\lambda (\mathcal{O})}\,d\lambda\\
    & \leq C \int_\R \Vert \hat{g}(\lambda)\Vert_{L^2(\Gamma_2)}^2\,d\lambda\\
    & \le C \|\widetilde g\|_{L^2(\R, L^2(\Gamma_2)} \\
    & \leq C \Vert g\Vert_{L^2((0,T),L^2(\Gamma_2))}^2.
\end{align*}
\end{proof}
We remark that the condition $C^5$ on the boundary $\partial\mathcal{O}$ is a technical requirement of Theorem 4.5 in \cite{DPRS23}, which is related to the localization procedure necessary to bring the problem from $\overline{\mathcal{O}}$ to the half-plane $\overline{\R^2_+}$.

\section{Lack of exponential stability}\label{exponential instability}

Here we will show that our system \eqref{ec1}-\eqref{initial2} is not exponentially stable if the Kelvin--Voigt damping term is absent. Our proof is supported by the following result which corresponds to Theorem 3.1 in \cite{Gomez2020Stability}.

\begin{theorem}\label{main tool}
Let $H_0$ be a closed subspace of a Hilbert space $H$. Let $\left(T_0(t)\right)_{t\in\R}$ be a unitary group on $H_0$ and $\left(T(t)\right)_{t\geq0}$ be a $C_0$-semigroup over $H$. If the difference $T(t)-T_0(t) : H_0\to H$ is a compact operator for all $t>0$, then  $\left(T(t)\right)_{t\geq0}$ is not exponentially stable.
\end{theorem}

We set $\HH_0\coloneqq \{0\}\times\{0\} \times H^1_0(\Omega_2)\times L^2(\Omega_2)\times\{0\}$ endowed with the norm $\left\|\hat w\right\|^2_{\HH_0}\coloneqq\be_2\left\|\na\hat w_3\right\|^2_{L^2(\Om_2)^2}+\rh_2\left\|\hat w_4\right\|^2_{L^2(\Om_2)}$ for $\hat w=(0, 0, \hat w_3, \hat w_4, 0)\in\HH_0$. Next, we consider the wave equation with zero Dirichlet boundary condition and suitable initial conditions,
\begin{equation}\label{system}
\begin{cases}
\rho_2\hat{v}_{tt}-\beta_2\Delta\hat{v}=0 \ \textup{ in } \ \R^+\times\Omega_2,\\
\hat{v}=0 \ \textup{ on } \ \R^+\times I,\\
\hat{v}(0, \cdot)=\hat{v}_0, \ \hat{v}_t(0, \cdot)=\hat{v}_1 \ \textup{ in } \ \Omega_2.
\end{cases}
\end{equation}
We introduce the operator
$$
\AA_0: D(\AA_0)\subset\HH_0\to\HH_0\textup{ given by } \AA_0\begin{pmatrix}
0\\
0\\
\hat w_3\\
\hat w_4\\
0
\end{pmatrix}
\coloneqq
\begin{pmatrix}
0\\
0\\
\hat w_4\\
\frac{\be_2}{\rh_2}\De\hat w_3\\
0
\end{pmatrix},
$$
where $D(\AA_0)\coloneqq \{0\}\times\{0\}\times H^2(\Omega_2)\cap H^1_0(\Omega_2)\times H^1_0(\Omega_2)\times\{0\}$. Putting $\hat w=(\hat w_j)^\top_{j=1, \ldots, 5}\coloneqq (0, 0, \hat v, \hat v_t, 0)^\top$ we have that the system \eqref{system} can be written as the Cauchy problem
$$
\pa_t\hat w(t)=\AA_0\hat w(t) \ (t>0) \;\; \textup{with} \;\; \hat w(0)=(0, 0, \hat v_0, \hat v_1, 0)^\top.
$$
It can be proven that $\AA_0$ generates a unitary group $(\mathscr T_0(t))_{t\in\R}$ on $\HH_0$, see e.g. \cite[p. 443]{bgh2023mana}. In the following lemma we obtain an estimate that relates the normal derivative of a solution of \eqref{system} with the initial conditions.

\begin{lemma}\label{normal derivative}
Let $\hat v$ be a sufficiently regular solution of \eqref{system} with $\hat v_0\in H^1(\Omega_2)$ and $\hat v_1\in L^2(\Omega_2)$. For $t>0$ the following estimate holds
\begin{equation}\label{estimate normal}
\left\|\pa_\nu\hat v\right\|^2_{L^2((0, t), L^2(I))}\leq C\big(\|\hat v_1\|^2_{L^2(\Om_2)}+\|\na\hat v_0\|^2_{L^2(\Om_2)^2}\big)
\end{equation}
with $C$ being a positive constant.
\end{lemma}

\begin{proof}
Let $t>0$. Multiplying the differential equation in \eqref{system} by the conjugate of $\hat v_t$, integrating over $\Om_2$, employing integration by parts and after taking real part yields
\begin{equation}\label{EQUAt1}
\frac{1}{2}\frac{d}{ds}\big(\rh_2\left\|\hat v_t(s)\right\|^2_{L^2(\Om_2)}+\be_2\left\|\na\hat v(s)\right\|^2_{L^2(\Om_2)^2}\big)=0
\end{equation}
for $0\leq s\leq t$. The initial data in \eqref{system} and \eqref{EQUAt1} imply
\begin{equation}\label{EqU7.3}
\int_{\Om_2}\rh_2|\hat v_t(s)|^2+\be_2|\na\hat v(s)|^2dx = \int_{\Om_2}\rh_2|\hat v_1|^2+\be_2|\na\hat v_0|^2dx.
\end{equation}
Let $h : \overline{\Om_2}\to\R^2$ be a vector field of class $C^1$ in $\overline{\Om_2}$ with $h\big|_I=-\nu$. Using the conditions that satisfy $h$ and $\hat v$ on $I$, theorem of divergence and the identity $h\cdot\na|\hat v_t(s)|^2=\div(h|\hat v_t(s)|^2)-(\div h)|\hat v_t(s)|^2$, we obtain
\begin{equation}\label{EqU7.4}
\begin{split}
\int_{\Om_2}\Re(\hat v_{tt}(s)h\cdot\overline{\na\hat v(s)})dx&=\frac{1}{2}\int_{\Om_2}(\div h)|\hat v_t(s)|^2dx\\
&\quad+\frac{d}{ds}\int_{\Om_2}\Re(\hat v_t(s)h\cdot\overline{\na\hat v(s)})dx.
\end{split}
\end{equation}
On the other side,  we have
\begin{equation}\label{EqU7.5}
\begin{split}
\int_{\Om_2}\Re(\De\hat v(s)h&\cdot\overline{\na\hat v(s)})dx=-\Re\int_{\Om_2}\sum\limits_{j,k=1}^2(\pa_j\hat v(s))(\pa_j h_k)(\overline{\pa_k\hat v(s)})\,dx\\
&\hspace{0.1555 cm}+\frac{1}{2}\int_{\Om_2}(\div h)|\na\hat v(s)|^2dx+\frac{1}{2}\int_I|\pa_\nu\hat v(s)|^2dS,
\end{split}
\end{equation}
where $dS$ stands for the surface measure on $I$. 
Multiplying the differential equation in \eqref{system} by $h\cdot\overline{\na\hat v}$ and using \eqref{EqU7.4} together with \eqref{EqU7.5}, we get
\begin{equation*}
\begin{split}
\int_I|\pa_\nu\hat v(s)|^2dS&=2\frac{\rh_2}{\be_2}\frac{d}{ds}\int_{\Om_2}\Re(\hat v_t(s)h\cdot\overline{\na\hat v(s)})dx+\frac{\rh_2}{\be_2}\int_{\Om_2}(\div h)|\hat v_t(s)|^2dx\\
&\quad+2\Re\int_{\Om_2}\sum\limits_{j,k=1}^2(\pa_j\hat v(s))(\pa_j h_k)(\overline{\pa_k\hat v(s)})\,dx \\
&\quad -\int_{\Om_2}(\div h)|\na\hat v(s)|^2dx.
\end{split}
\end{equation*}
Note that $|h|$, $|\div h|$ and $|\pa_j h_k|$ are bounded scalar fields on $\overline{\Om_2}$. Integrating above with respect to the variable $s$ from 0 to $t$, and then using the inequalities of Cauchy--Schwarz and Young together with equality \eqref{EqU7.3}, we deduce \eqref{estimate normal}.
\end{proof}

Recall that $\AA : D(\AA)\subset\HH\to\HH$ is a densely defined closed linear operator with $0\in\rh(\AA)$, see Theorem \ref{C_0 semigroup} and Lemma \ref{Prop_IR}. The normed space $\big(\HH, \left\|\AA^{-1}\cdot\right\|_{\HH}\big)$ has a completion $\left(\HH_{-1},\left\| \; \cdot \; \right\|_{-1}\right)$, where $\left\| \; \cdot \; \right\|_{-1}\ceqq\left\|\AA^{-1}\cdot\right\|_{\HH}$, called the \it{extrapolation space} of $\HH$ generated by $\AA$. It is known that for the adjoint operator $\AA' : D(\AA')\subset\HH'\to\HH'$ one has that $\HH_{-1}=[D(\AA')]'$, see Corollary 1.4.7 in \cite[p. 271]{amann1995linear}. Below, we present the main result of this section.  In the proof, Lemma \ref{Prop 3.3} is used. Therefore, we will assume that $\partial\Omega_1$ is of class $C^5$.


\begin{theorem}\label{non-exponencial}
\textcolor{black}{Suppose that $\Omega_1$ has a $C^5$-boundary $\partial\Omega_1$}. For $m=0$, we have that the system \eqref{ec1}-\eqref{initial2} does not have exponential decay.
\end{theorem}
\begin{proof}
It is very simple to see that $\HH_0$ is a closed subspace of $\HH$. We will show that $\mathscr T(t)-\mathscr T_0(t) \colon\HH_0\to\HH$ is a compact operator for all $t>0$. It is enough to prove that for any $t>0$ the operator $\mathscr T(t)-\mathscr T_0(t) : \mathscr D\to\HH$ is compact because $\mathscr D\coloneqq  \{0\}^2\times[\mathscr D(\Omega_2)]^2\times\{0\}$ is a dense subspace of $\HH_0$. For $w_0\in\mathscr D$ and $t>0$, we set
\begin{equation}\label{eq4.3,1}
\mathscr E(t)\coloneqq  \frac{1}{2}\|\mathscr T(t)w_0-\mathscr T_0(t)w_0\|^2_{\HH}.
\end{equation}
Note that $D(\AA)\ni w(t) \coloneqq  \mathscr T(t)w_0$ and $D(\AA_0)\ni \hat w(t) \coloneqq  \mathscr T_0(t)w_0$ because $w_0\in D(\AA)\cap D(\AA_0)$. By \eqref{abstract with eta=0} and \eqref{eq4.3,1}, we can write
\begin{equation}\label{5-1}
\begin{split}
\frac{d}{dt}\mathscr E(t)&=\Re\left(\AA w(t),w(t)\right)_{\HH} + \Re(\AA_0 \hat w(t),\hat w(t))_{\HH}\\
&\quad-\Re\left(\AA w(t),\hat w(t)\right)_{\HH}-\Re(\AA_0 \hat w(t), w(t))_{\HH}.
\end{split}
\end{equation}
From \eqref{dissipativity} we know $ \Re\left(\AA w(t),w(t)\right)_{\HH} \le 0$. Using integration by parts and taking into account that $\hat w_4(t)=0$ on $I$, we get the following expression
\begin{align*}
(\AA_0 \hat w(t),\hat w(t))_{\HH}=i2\be_2\Im\left(\na\hat w_4(t), \na\hat w_3(t)\right)_{L^2(\Om_2)^2}
\end{align*}
and thus $\Re(\AA_0 \hat w(t),\hat w(t))_{\HH}=0$. From the definition of the operator $\AA$, it is immediate that
$$
\left(\AA w(t),\hat w(t)\right)_{\HH}=\be_2(\na w_4(t), \na\hat w_3(t))_{L^2(\Om_2)^2}-\be_2(\na w_3(t), \na\hat w_4(t))_{L^2(\Om_2)^2}.
$$
Employing integration by parts, we obtain
\begin{align*}
&(\AA_0 \hat w(t), w(t))_{\HH}=\be_2(\na\hat w_4(t), \na w_3(t))_{L^2(\Om_2)^2}\\
&-\be_2(\na\hat w_3(t), \na w_4(t))_{L^2(\Om_2)^2}-\be_2(\pa_\nu\hat w_3(t), w_4(t))_{L^2(I)}.
\end{align*}
The last two equalities produce the following expression
\begin{align*}
&-\left(\AA w(t),\hat w(t)\right)_{\HH}-(\AA_0 \hat w(t), w(t))_{\HH}=i2\be_2\Im(\na\hat w_3(t), \na w_4(t))_{L^2(\Om_2)^2}\\
&+i2\be_2\Im(\na w_3(t), \na\hat w_4(t))_{L^2(\Om_2)^2}+\be_2(\pa_\nu\hat w_3(t), w_4(t))_{L^2(I)}.
\end{align*}
Taking real part in the last equality, inserting this into \eqref{5-1} and integrating over $(0, t)$, we obtain the estimate
\begin{equation}\label{integral}
\mathscr E(t)\le \beta_2\Re \left(\pa_\nu\hat w_3, w_2\right)_{L^2((0, t), L^2(I))}.
\end{equation}
The equality $\mathscr E(0)=0$ was used and the fact that $w_2(t)=w_4(t)$ on $I$.

Let $(w_0^{k})_{k\in\N}\subset\mathscr D$ be a bounded sequence in $\HH_0$. We must prove that $(\mathscr T(t)w^k_0-\mathscr T_0(t)w^k_0)_{k\in\N}$ possesses a convergent subsequence in $\HH$. This would show the compactness of $\mathscr T(t)-\mathscr T_0(t) \colon \mathscr D\to\HH$. We set $w^{k}(t) \coloneqq  \mathscr T(t)w_0^{k}$ and $\hat w ^{k}(t) \coloneqq\mathscr T_0(t)w_0^{k}$. Due to Lemma \ref{normal derivative},
\begin{equation}\label{EQ7.4}
\|\pa_\nu\hat w_3^k\|_{L^2((0, t), L^2(I))}\leq C\|\hat w^k(0)\|_{\HH_0}=C\|w^k_0\|_{\HH_0}\leq C\;\;\;\textup{for all}\;\;\;k\in\N.
\end{equation}
Let $\phi\in\Yy\ceqq \{0\}\times H^2_0(\Om_1)\times \{0\}\times H^1_0(\Om_2)\times \{0\}$.  Taking $w\in D(\AA)$,   we get   by straightforward calculation, using integration by parts,
\begin{align*}
\left(\AA w, \phi\right)_{\HH}&=\rho_1\Big(-\frac{\beta_1}{\rho_1}\Delta^2w_1 - \frac{\alpha}{\rho_1}w_5, \phi_2\Big)_{L^2(\Om_1)} + \rho_2\Big(\frac{\beta_2}{\rho_2}\Delta w_3 , \phi_4\Big)_{L^2(\Omega_2)}\\
& = (w,z_\phi)_{\HH},
\end{align*}
where $z_\phi\ceqq (-\phi_2, 0, -\phi_4,0,-\frac{\alpha}{\rho_0}\Delta \phi_2)^\top\in \HH$.
Therefore,
$$
\big|\left(\AA w, \phi\right)_{\HH}\big|\leq C_\phi\left\|w\right\|_{\HH},
$$
where $C_\phi\coloneqq \left\|z_\phi\right\|_\HH$. In consequence, the  linear functional $\psi_\phi\colon D(\AA)\to \C$ defined by 
$\psi_\phi w\ceqq\left(\AA w, \phi\right)_{\HH}$ is continuous 
with respect to the norm $\|\cdot\|_{\HH}$. As $D(\AA)$ is dense in $\HH$, there exists a unique  linear and continuous extension of $\psi_\phi$ to $\HH$, which is denoted by $\tilde \psi_\phi  : \HH\to\C$. Note that, by means of  the identifications provided via the  isometric Riesz  isomorphism, $\phi\in\HH'$, and for $ w\in D(\AA)$ we have
\begin{align*}
\langle w, \tilde \psi_\phi \rangle_{\HH\times \HH'}&=\tilde \psi_\phi  w=\psi_\phi  w=\left(\AA w, \phi\right)_{\HH}=\langle\AA w, \phi\rangle_{\HH\times \HH'}\\
 & = (w,z_\phi)_\HH = \langle w, z_\phi\rangle_{\HH\times \HH'}.
\end{align*}
Thus, $\phi\in D(\AA')$. With this, it is shown that $\Yy\subset D(\AA')$ as sets, and that  $\AA'\phi = z_\phi$.\\
Now, for $\phi\in \Yy$ it holds
\begin{align*}
    \Vert \phi\Vert_{D(\AA')}^2 & \leq C\big( \Vert \phi\Vert_{\HH'}^2 + \Vert \AA'\phi\Vert_{\HH'}^2\big)\\
    & = C\big( \Vert \phi\Vert_{\HH'}^2 + \Vert z_\phi\Vert_{\HH'}^2\big)\\
    & = C\big( \Vert \phi\Vert_{\HH}^2 + \Vert z_\phi\Vert_{\HH}^2\big)\\
    & = C\Big( \rho_1\Vert \phi_2\Vert_{L^2(\Omega_1)}^2 + \rho_2\Vert \phi_4\Vert_{L^2(\Omega_2)}^2 + \beta_1\Vert \phi_2\Vert_{H^2_\Gamma(\Omega_1)}^2 + \beta_2\Vert \nabla \phi_4\Vert_{L^2(\Omega_2)^2}^2\\
    & \qquad \qquad + \frac{\alpha^2}{\rho_0^2}\Vert \Delta \phi_2\Vert_{L^2(\Omega_1)}^2\Big)\\
    & \leq C\Vert \phi\Vert_{\Yy}^2,
\end{align*}
where $\Vert \cdot\Vert_{\Yy}$ is the usual norm in the product space that defines $\Yy$.

For $t>0$ and $0\leq s\leq t$, $\partial_s w^k(s)$ can be considered as an element of $D(\AA')'$ and for $\phi\in\Yy$ we have
$$\vert \partial_s w^k(s)(\phi)\vert \leq \Vert \partial_sw^k(s)\Vert_{D(\AA')'}\Vert \phi\Vert_{D(\AA')}\leq C\Vert\partial_sw^k(s)\Vert_{D(\AA')'}\Vert \phi\Vert_{\Yy}.$$
Then, 
$$\Vert \partial_sw^k(s)\Vert_{\Yy'}\leq C\Vert\partial_sw^k(s)\Vert_{D(\AA')'}.$$
Therefore,
\begin{align*}
    \sup_{s\in[0,t]}\|\pa_sw_2^k(s)\|_{H^{-2}(\Omega_1)} & \leq C\sup_{s\in[0,t]}\|\pa_sw^k(s)\|_{\Yy'}\leq C\sup_{s\in[0,t]}\|\pa_sw^k(s)\|_{D(\AA')'}\\
    & \leq C\sup_{s\in[0,t]}\|\pa_sw^k(s)\|_{\HH_{-1}} =C\sup_{s\in [0,t]}\|\AA w^k(s)\|_{\HH_{-1}}\\
    & = C\sup_{s\in[0,t]}\|w^k(s)\|_{\HH}\leq C
\end{align*}
for all $k\in\N$. It follows that $(\pa_tw^k_2)_{k\in\N}$ is bounded in $L^2((0,t), H^{-2}(\Om_1))$.

 Note that $\widetilde{w}^k:=(w_1^k,w_2^k,w_5^k)^\top$ satisfies
\begin{align*}
\partial_t \widetilde{w}^k - A(D)\widetilde{w}^k & = (0,0,0)^\top & \text{in}\ (0,\infty)\times \Omega_1,\\
B_1(D)\widetilde{w}^k & = (0,0,0)^\top & \text{on}\ (0,\infty)\times \Gamma_1,\\
B_2(D)\widetilde{w}^k & = (0,-\beta_2\partial_\nu w_3^k, 0)^\top & \text{on}\ (0,\infty)\times \Gamma_2.
\end{align*}
Applying Lemma \ref{Prop 3.3}, one can establish that
$$
\|w_2^k\|_{L^2((0, t), H^{3/2}(\Om_1))}\leq C\|\pa_\nu w_3^k\|_{L^2((0, t), L^2(I))}\;\;\;\textup{for all}\;\;\;k\in\N\,,
$$
and so \eqref{EQ7.4} implies that $(w^k_2)_{k\in\N}$ is a bounded sequence in $L^2((0,t), H^{3/2}(\Om_1))$. The  embeddings $H^{3/2}(\Om_1)\overset{c}{\hookrightarrow}H^1(\Om_1)\hookrightarrow H^{-2}(\Om_1)$ and the Aubin--Lions lemma (Lemma~\ref{A.4} in the Appendix) imply the existence of a subsequence $(w^{k_j}_2)_{j\in\N}$ of $(w^k_2)_{k\in\N}$ which is convergent in $L^2((0,t), H^1(\Om_1))$. Taking
the trace on $I$, we obtain convergence in $L^2((0,t),L^2(I))$ for the subsequence $(w^{k_j}_2)_{j\in\N}$ of $(w^{k}_2\big|_{I})_{k\in\N}$.

Note that  $((\mathscr T(t)-\mathscr T_0(t))w^{k_j}_0)_{j\in\N}$ is a subsequence of $(\mathscr T(t)w^k_0-\mathscr T_0(t)w^k_0)_{k\in\N}$. We write $a^{nm}\coloneqq a_n-a_m$ for any sequence $(a_n)_{n\in\N}$. With this notation, we have
\begin{equation}\label{EQ.7}
w^{k_ik_j}(t)-\hat w^{k_ik_j}(t)=\mathscr T(t)(w_0^{k_i}-w_0^{k_j})-\mathscr T_0(t)(w_0^{k_i}-w_0^{k_j}).
\end{equation}
For $i, j\in\N$ and $t\geq0$, we now consider
\begin{equation}\label{EQ.8}
\mathscr E^{ij}(t)\coloneqq \frac{1}{2}\|w^{k_ik_j}(t)-\hat w^{k_ik_j}(t)\|^2_{\HH}.
\end{equation}
By \eqref{integral} and \eqref{EQ7.4}, we compute that
\begin{align}\label{EQ.9}
\mathscr E^{ij}(t)&\leq\be_2|(\pa_\nu\hat w_3^{k_ik_j}, w_2^{k_ik_j})_{L^2((0, t), L^2(I))}|\nonumber\\
&\leq\be_2\|\pa_\nu\hat w_3^{k_ik_j}\|_{L^2((0, t), L^2(I))}\|w_2^{k_ik_j}\|_{L^2((0, t), L^2(I))}\nonumber\\
&\leq C\|w_2^{k_i}-w_2^{k_j}\|_{L^2((0, t), L^2(I))}\to0 \;\;\; (\textup{as } i, j\to\infty).
\end{align}
Thanks to \eqref{EQ.7}-\eqref{EQ.9}, we get that $((\mathscr T(t)-\mathscr T_0(t))w_0^{k_j})_{j\in\N}$ is a Cauchy sequence in $\HH$ and therefore converges in this Hilbert space. Accordingly, $\mathscr T(t)-\mathscr T_0(t)$ is a compact operator from $\HH_0$ to $\HH$ for any $t>0$. Now, Theorem \ref{main tool} leads to the conclusion of the present theorem.
\end{proof}

\section{\textcolor{black}{Polynomial stability without geometric condition}}\label{Polynomial stability}

Here we will see that the system \eqref{ec1}-\eqref{initial2} is polynomially stable without any geometric condition imposed on the boundary of $\Om_2$, when there is no damping on the membrane. In the first subsection, we will show a result which is the key to prove, by contradiction, the polynomial stability of a $C_{0}$-semigroup of contractions. In the second subsection, we will prove the polynomial stability if $m=0$. The main result of this section is the following:

\begin{theorem}\label{Th poly without GC}
Let $m=0$. Then the semigroup
$\left(\mathscr{T}(t)\right)_{t\geq0}$ decays
polynomially. 
\end{theorem}

\smallskip

\subsection{A characterization of polynomial stability}\quad\\
With respect to polynomial decay, we have the following characterization of the polynomial stability of a bounded $C_0$-semigroup due to Borichev and Tomilov (see \cite[Theorem 2.4]{Borichev2010optimal}).

\begin{theorem} \label{Borichev-Tomilov}
Let $(T(t))_{t\geq0}$ be a bounded $C_0$-semigroup on a Hilbert space $H$ with generator $A$ such that $i\mathbb{R}\cap \sigma(A)$ is empty. Then, for  $r>0$ fixed, the following assertions are equivalent:
\begin{itemize}
\item[i)] There exist $C>0$ and $\lambda_0>0$ such that for all $\lambda\in\mathbb{R}$ with $\vert \lambda\vert > \lambda_0$ and all $f\in H$ it holds
$$ \big\Vert (i\lambda I - A)^{-1}f\big\Vert_H \leq C \vert \lambda\vert^{r}\Vert f\Vert_H.$$
\item[ii)] There exists some $C>0$ such that for all $t>0$ it holds
$$ \big\Vert T(t)A^{-1}\big\Vert_{\mathscr{L}(H)}\leq C t^{-\,\frac{1}{r}}.$$
\end{itemize}
\end{theorem}

\begin{remark}\label{Prop sequences PolyStab}
    Let the hypotheses of Theorem \ref{Borichev-Tomilov} be fulfilled. Note that the inequality in the assertion i) of Theorem \ref{Borichev-Tomilov} is equivalent to
    $$ \vert \lambda\vert^{-r}\big\Vert (i\lambda I - A)^{-1}\big\Vert_{\mathscr{L}(H)}\leq C \qquad (\lambda\in\R, \ \vert\lambda\vert>\lambda_0),$$
which in turn is also equivalent to
$$ \limsup\limits_{\vert\lambda\vert\to\infty}\,\vert \lambda\vert^{-r}\big\Vert (i\lambda I - A)^{-1}\big\Vert_{\mathscr{L}(H)} < \infty.$$
Now, suppose that
$$\limsup\limits_{\vert\lambda\vert\to\infty}\vert \lambda\vert^{-r}\big\Vert (i\lambda I - A)^{-1}\big\Vert_{\mathscr{L}(H)}=\infty.$$
In this case, there are sequences $(\lambda_n)_{n\in\mathbb{N}}\subset\mathbb{R}$ and    $(w_n)_{n\in\mathbb{N}}\subset D(A)$, such that  $\vert \lambda_n\vert\to\infty$ as $n\to\infty$,   $\Vert w_n\Vert_H=1$ for all $n\in\mathbb{N}$, and
$$\vert \lambda_n\vert^{r}\big\Vert (i\lambda_n I - A)w_n\big\Vert_H\to 0\quad \text{as}\quad n\to\infty.$$
In fact, for each $n\in\mathbb{N}$ it holds
$$ \sup_{\vert \lambda\vert \geq n} \vert \lambda\vert^{-r}\big\Vert (i\lambda I - A)^{-1}\big\Vert_{\mathscr{L}(H)} > n. $$
Then, for each $n\in\mathbb{N}$, there exists $\lambda_n\in\mathbb{R}$ with $\vert \lambda_n\vert \geq n$ and
$$  \vert \lambda_n\vert^{-r}\big\Vert (i\lambda_n I - A)^{-1}\big\Vert_{\mathscr{L}(H)} > n. $$
Note that $\vert\lambda_n\vert\to \infty$ as $n\to\infty$.  The last inequality implies that for each $n\in\mathbb{N}$ there exists $h_n\in H\smallsetminus\{0\}$ such that
$$ \big\Vert (i\lambda_n I - A)^{-1}h_n\big\Vert_H > n\vert \lambda_n\vert^{r}\Vert h_n\Vert_H.$$
Let $\widetilde{w}_n \ceqq (i\lambda_n I - A)^{-1}h_n $ for $n\in\mathbb{N}$. Then, $\widetilde{w}_n\in D(A)\smallsetminus\{0\}$. Moreover,
$$ \Vert \widetilde{w}_n\Vert_H > n \vert \lambda_n\vert^{r}\big\Vert (i\lambda_n I - A)\widetilde{w}_n\big\Vert_H,$$
which is equivalent to
$$ \frac{1}{n} > \vert \lambda_n\vert^{r}\big\Vert (i\lambda_n I - A)w_n\big\Vert_H, $$
where $w_n\ceqq\dfrac{\widetilde{w}_n}{\Vert \widetilde{w}_n\Vert_H}$ for $n\in\mathbb{N}$. The sequence $(w_n)_{n\in\mathbb{N}}$ satisfies $w_n\in D(A)$ and $\Vert w_n\Vert_H=1$ for all $n\in\mathbb{N}$. Furthermore,
$$\vert \lambda_n\vert^{r}\big\Vert (i\lambda_n I - A)w_n\big\Vert_H\to 0 \quad (n\to\infty).$$
\end{remark}

\smallskip

\subsection{Proof of Theorem \ref{Th poly without GC}} \quad \\
In the following, let $m=0$. In this case, the norm in $\mathscr{H}$ is given by%
\begin{align*}
\left\Vert w\right\Vert _{\mathscr{H}}^{2}& =\beta_{1}\left\Vert w_{1}%
\right\Vert _{H_{\Gamma}^{2}\left(  \Omega_{1}\right)  }^{2}+\rho
_{1}\left\Vert w_{2}\right\Vert _{L^{2}\left(  \Omega_{1}\right)  }^{2}%
+\beta_{2}\left\Vert \nabla w_{3}\right\Vert _{L^{2}\left(  \Omega_{2}\right)
}^{2}+\rho_{2}\left\Vert w_{4}\right\Vert _{L^{2}\left(  \Omega_{2}\right)
}^{2}\\
& \quad +\rho_{0}\left\Vert w_{5}\right\Vert _{L^{2}\left(  \Omega_{1}\right)
}^{2}\text{.}%
\end{align*}


We will first prove the theorem under the additional assumption that $i\R\subset\rho(\AA)$, which
 will then be shown in Lemma~\ref{Prop_resolvent}. 
 
We fix a sufficiently large $r\in \N$ and prove polynomial stability with rate $1/r$ by an indirect proof. In fact, we will obtain $r=192$.  If the semigroup generated by $\mathscr A$ is not polynomially stable with rate $r$, then, under the assumption that $i\R\subset\rho(\AA)$ and
by Remark~\ref{Prop sequences PolyStab},  there are sequences $\left(  \lambda
_{n}\right)  _{n\in\mathbb{N}}\subset\mathbb{R}$ and $\left(  w^{n}\right)
_{n\in\mathbb{N}}\subset D(\mathscr{A})$ such that%
$$
\left\vert \lambda_{n}\right\vert \rightarrow\infty\text{,\quad}\left\Vert
w^{n}\right\Vert _{\mathscr{H}}=1\text{ for all }n\text{,\quad and \quad
}\left\vert \lambda_{n}\right\vert ^{r}\left\Vert \left(  i\lambda
_{n}I-\mathscr{A}\right)  w^{n}\right\Vert _{\mathscr{H}}\rightarrow0\text{.}%
$$
Let%
\begin{equation}
\widetilde{f}^{n}\ceqq\mu_{n}^{r}\left(  i\lambda_{n}I-\mathscr{A}\right)
w^{n}\label{ec resolvent2}%
\end{equation}
with $\mu_{n}\ceqq\left\vert \lambda_{n}\right\vert $. Then we have $\widetilde{f}^n\to 0$ in $\HH$. In the following, we will show  that this leads to the contradiction $\|w^n\|_{\HH}\to 0$, 
by estimating the components of $w^n$. This is done in several steps.

\smallskip
\noindent\textbf{(i)} \textit{Estimate of $w_5^n$: }
By the dissipation \eqref{dissipativity}, 
we have
\begin{align*}
\mu_{n}^{r}\left\Vert w_{5}^{n}\right\Vert _{H^{1}\left(  \Omega_{1}\right)
}^{2}  &  \leq C\mu_{n}^{r}\left(  \sigma\left\Vert w_{5}^{n}\right\Vert
_{L^{2}\left(  \Omega_{1}\right)  }^{2}+\beta\left\Vert \nabla w_{5}%
^{n}\right\Vert _{L^{2}\left(  \Omega_{1}\right)  ^{2}}^{2}+k\beta\left\Vert
w_{5}^{n}\right\Vert _{L^{2}\left(  \partial\Omega_{1}\right)  }^{2}\right) \\
&  = - C\mu_{n}^{r}\operatorname{Re}\left(  \mathscr{A}w^{n},w^{n}\right)
_{\mathscr{H}}\\
&  = C\operatorname{Re}\, (  \mu_{n}^{r}(  i\lambda_{n}I-\mathscr{A})  w^{n},w^{n})_{\mathscr{H}}\\
&  = C\Re\, (\widetilde{f}^{n},w^{n} )_{\mathscr{H}}\\
&  \leq C\Vert \widetilde{f}^{n}\Vert _{\mathscr{H}}\text{.}%
\end{align*}
As $\tilde{f}^n\to 0$, this yields
\begin{equation}
\mu_{n}^{r/2}\left\Vert w_{5}^{n}\right\Vert _{H^{1}\left(  \Omega_{1}\right)
}\leq C \Vert \widetilde{f}^{n}\Vert _{\mathscr{H}}^{1/2}\text{
\quad and so\quad}\mu_{n}^{r/2}\left\Vert w_{5}^{n}\right\Vert _{H^{1}\left(
\Omega_{1}\right)  }\rightarrow0\text{.} \label{estimate dissip}%
\end{equation}

\smallskip
\noindent\textbf{(ii)} \textit{Estimate of $\|w_1^n\|_{L^2(\Omega_1)}$:}
From (\ref{ec resolvent2}), it follows that%
\begin{align}
\mu_{n}^{r}\left(  i\lambda_{n}w_{1}^{n}-w_{2}^{n}\right)   &  =\widetilde{f}%
_{1}^{n} & \!\!\!\!\!\!\!\!\!\!\!\!\rightarrow0\text{ in }H_{\Gamma}^{2}\left(  \Omega_{1}\right)
\text{,}\label{ast1}\\
\mu_{n}^{r}\left(  i\rho_{1}\lambda_{n}w_{2}^{n}+\beta_{1}\Delta^{2}w_{1}%
^{n}+\alpha\Delta w_{5}^{n}\right)   &  =\rho_1 \widetilde{f}_{2}^{n} & \!\!\!\!\!\!\!\!\!\!\!\!\rightarrow
0\text{ in }L^{2}\left(  \Omega_{1}\right)  \text{,}\label{ast2}\\
\mu_{n}^{r}\left(  i\lambda_{n}\nabla w_{3}^{n}-\nabla w_{4}^{n}\right)   &
=\nabla\widetilde{f}_{3}^{n} & \!\!\!\!\!\!\!\!\!\!\!\!\rightarrow0\text{ in }L^{2}\left(  \Omega
_{2}\right)  \text{,}\label{ast3}\\
\mu_{n}^{r}\left(  i\rho_{2}\lambda_{n}w_{4}^{n}-\beta_{2}\Delta w_{3}%
^{n}\right)   &  =\rho_2 \widetilde{f}_{4}^{n} & \!\!\!\!\!\!\!\!\!\!\!\!\rightarrow0\text{ in }L^{2}\left(
\Omega_{2}\right)  \text{,}\label{ast4}\\
\mu_{n}^{r}\left(  i\rho_{0}\lambda_{n}w_{5}^{n}-\alpha\Delta w_{2}^{n}%
-\beta\Delta w_{5}^{n}+\sigma w_{5}^{n}\right)   &  =\rho_0 \widetilde{f}_{5}%
^{n} & \!\!\!\!\!\!\!\!\!\!\!\!\rightarrow0\text{ in }L^{2}\left(  \Omega_{1}\right)  \text{.}
\label{ast5}%
\end{align}
Due to (\ref{ast1}) we have $w_{2}^{n}=i\lambda_{n}w_{1}^{n}-\mu_{n}%
^{-r}\widetilde{f}_{1}^{n}$. Inserting this in (\ref{ast2}), we get%
$$
i\rho_{1}\lambda_{n}\left(  i\lambda_{n}w_{1}^{n}-\mu_{n}^{-r}\widetilde{f}%
_{1}^{n}\right)  +\beta_{1}\Delta^{2}w_{1}^{n}+\alpha\Delta w_{5}^{n}=\rho_1 \mu
_{n}^{-r}\widetilde{f}_{2}^{n}\text{,}%
$$
i.e.,%
\begin{equation}
-\rho_{1}\lambda_{n}^{2}w_{1}^{n}+\beta_{1}\Delta^{2}w_{1}^{n}=\mu_{n}%
^{-r}\widetilde{f}_{2}^{n}+i\rho_{1}\lambda_{n}\mu_{n}^{-r}\widetilde{f}%
_{1}^{n}-\alpha\Delta w_{5}^{n}\text{.} \label{ast6}%
\end{equation}
Now, from $\mu_{n}^{r}\left(  i\lambda_{n}w_{3}^{n}-w_{4}^{n}\right)
=\widetilde{f}_{3}^{n}$, we have that
\begin{equation}
w_{4}^{n}=i\lambda_{n}w_{3}^{n}-\mu_{n}^{-r}\widetilde{f}_{3}^{n}\text{.}
\label{ast3b}%
\end{equation}
Inserting this in (\ref{ast4}), we obtain%
$$
i\rho_{2}\lambda_{n}(  i\lambda_{n}w_{3}^{n}-\mu_{n}^{-r}\widetilde{f}_{3}^{n})  -\beta_{2}\Delta w_{3}^{n}=\mu_{n}^{-r}\widetilde{f}_{4}%
^{n}\text{,}%
$$
i.e.,%
\begin{equation}
-\rho_{2}\lambda_{n}^{2}w_{3}^{n}-\beta_{2}\Delta w_{3}^{n}=\mu_{n}%
^{-r}\widetilde{f}_{4}^{n}+i\rho_{2}\lambda_{n}\mu_{n}^{-r}\widetilde{f}%
_{3}^{n}\text{.} \label{ast7}%
\end{equation}
Multiplying (\ref{ast6}) by $-\overline{w_{1}^{n}}$ in $L^{2}\left(
\Omega_{1}\right)  $, we get%
\begin{multline*}
\rho_{1}\lambda_{n}^{2}\left\Vert w_{1}^{n}\right\Vert _{L^{2}\left(
\Omega_{1}\right)  }^{2}-\beta_{1}\left(  \Delta^{2}w_{1}^{n},w_{1}%
^{n}\right)  _{L^{2}\left(  \Omega_{1}\right)  }\\
=-\mu_{n}^{-r}(\widetilde{f}_{2}^{n},w_{1}^{n})_{L^{2}\left(  \Omega_{1}\right)}
-i\rho_{1}\lambda_{n}\mu_{n}^{-r}(  \widetilde{f}_{1}^{n},w_{1}^{n})_{L^{2}\left(  \Omega_{1}\right)  }+\alpha\left(  \Delta
w_{5}^{n},w_{1}^{n}\right)  _{L^{2}\left(  \Omega_{1}\right)  }\text{.}%
\end{multline*}
Then, by integration by parts and the transmission conditions, it holds 
\begin{align*}
&  \rho_{1}\lambda_{n}^{2}\left\Vert w_{1}^{n}\right\Vert _{L^{2}\left(
\Omega_{1}\right)  }^{2}\\
& \quad -\Big[  \beta_{1}\left\Vert w_{1}^{n}\right\Vert
_{H_{\Gamma}^{2}\left(  \Omega_{1}\right)  }^{2}-\big(  \underbrace{\beta
_{1}\mathscr{B}_{1}w_{1}^{n}}_{=-\alpha w_{5}^{n}},\partial_{\nu}w_{1}%
^{n}\big)_{L^{2}\left(  I\right)  }+\big(  \!\!\!\!\!\!\underbrace{\beta
_{1}\mathscr{B}_{2}w_{1}^{n}}_{=\alpha\kappa w_{5}^{n}-\beta_{2}\partial_{\nu
}w_{3}^{n}} \!\!\!\!\!\!, w_{1}^{n}\big)_{L^{2}\left(  I\right)  } \Big] \\
&  =-\mu_{n}^{-r}(  \widetilde{f}_{2}^{n},w_{1}^{n})_{L^{2}\left(  \Omega_{1}\right)}-i\rho_{1}\lambda_{n}\mu_{n}^{-r}(\widetilde{f}_{1}^{n},w_{1}^{n})_{L^{2}\left(\Omega_{1}\right)}-\alpha\left(  \nabla w_{5}^{n},\nabla w_{1}^{n}\right)_{L^{2}\left(
\Omega_{1}\right)  }\\
& \quad + \alpha (  \underbrace{\partial_{\nu}w_{5}^{n}}_{=-\kappa w_{5}^{n}},w_{1}^{n})_{L^{2}\left(  I\right)  }\text{.}%
\end{align*}
From this, we obtain
\begin{multline}
\beta_{1}\left\Vert w_{1}^{n}\right\Vert _{H_{\Gamma}^{2}\left(  \Omega
_{1}\right)  }^{2}    = \rho_{1}\lambda_{n}^{2}\left\Vert w_{1}^{n}\right\Vert
_{L^{2}\left(  \Omega_{1}\right)  }^{2} - \alpha\left(  w_{5}^{n},\partial_{\nu
}w_{1}^{n}\right)  _{L^{2}\left(  I\right)  }\\ + \beta_{2}\left(  \partial_{\nu
}w_{3}^{n},w_{1}^{n}\right)  _{L^{2}\left(  I\right)  }
+\mu_{n}^{-r}(\widetilde{f}_{2}^{n},w_{1}^{n})  _{L^{2}\left(  \Omega_{1}\right)
}\\
+ i\rho_{1}\lambda_{n}\mu_{n}^{-r}(\widetilde{f}_{1}^{n},w_{1}^{n})_{L^{2}\left(  \Omega_{1}\right)  } + \alpha\left(  \nabla w_{5}^{n},\nabla w_{1}^{n}\right)  _{L^{2}\left(  \Omega_{1}\right)  }\text{.}
\label{ast8}%
\end{multline}
Now, multiplying (\ref{ast7}) by $-\overline{w_{3}^{n}}$ in $L^{2}\left(\Omega_{2}\right)$ and using integration by parts, we get%
\begin{multline*}
\rho_{2}\lambda_{n}^{2}\left\Vert w_{3}^{n}\right\Vert_{L^{2}(\Omega_{2})}^{2}+\beta_{2}\big[  -\left\Vert \nabla w_{3}^{n}\right\Vert_{L^{2}(\Omega_{2})}^{2}-(  \partial_{\nu}w_{3}^{n},\underbrace{w_{3}^{n}}_{=w_{1}^{n}})_{L^{2}\left( I \right)}  \big]\\
  =-\mu_{n}^{-r}(  \widetilde{f}_{4}^{n},w_{3}^{n})_{L^{2}\left(  \Omega
_{2}\right)  }-i\rho_{2}\lambda_{n}\mu_{n}^{-r}(  \widetilde{f}_{3}^{n},w_{3}^{n})_{L^{2}\left(  \Omega_{2}\right)  }\text{.}%
\end{multline*}
Then,%
\begin{multline}
\beta_{2}\left\Vert \nabla w_{3}^{n}\right\Vert _{L^{2}\left(  \Omega
_{2}\right)  }^{2}=\rho_{2}\lambda_{n}^{2}\left\Vert w_{3}^{n}\right\Vert
_{L^{2}\left(  \Omega_{2}\right)  }^{2}-\beta_{2}\left(  \partial_{\nu}w_{3}^{n},w_{1}^{n}\right)_{L^{2}\left( I \right)}\\
+\mu_{n}^{-r}(  \widetilde{f}_{4}^{n},w_{3}^{n})_{L^{2}\left(  \Omega_{2}\right)  }
+i\rho_{2}\lambda_{n}\mu_{n}^{-r}(  \widetilde{f}_{3}^{n},w_{3}^{n})_{L^{2}\left(
\Omega_{2}\right)  }\text{.} \label{ast9}%
\end{multline}
Adding (\ref{ast8}) and (\ref{ast9}), we obtain%
\begin{align}
&  \beta_{1}\left\Vert w_{1}^{n}\right\Vert _{H_{\Gamma}^{2}\left(  \Omega
_{1}\right)  }^{2}+\beta_{2}\left\Vert \nabla w_{3}^{n}\right\Vert
_{L^{2}\left(  \Omega_{2}\right)  }^{2}\nonumber\\
&  =\rho_{1}\lambda_{n}^{2}\left\Vert w_{1}^{n}\right\Vert _{L^{2}\left(
\Omega_{1}\right)  }^{2}-\alpha\left(  w_{5}^{n},\partial_{\nu}w_{1}%
^{n}\right)  _{L^{2}\left(  I\right)  }+\mu_{n}^{-r}(  \widetilde{f}_{2}^{n},w_{1}^{n})  _{L^{2}\left(  \Omega_{1}\right)  } \nonumber\\
& \quad + i\rho_{1}\lambda_{n}\mu_{n}^{-r}(  \widetilde{f}_{1}^{n},w_{1}^{n})
_{L^{2}\left(  \Omega_{1}\right)  }
  + \alpha\left(  \nabla w_{5}^{n},\nabla w_{1}^{n}\right)_{L^{2}\left(\Omega_{1}\right)} +\rho_{2}\lambda_{n}^{2}\left\Vert
w_{3}^{n}\right\Vert _{L^{2}\left(  \Omega_{2}\right)  }^{2} \nonumber\\
& \quad + \mu_{n}^{-r}(  \widetilde{f}_{4}^{n},w_{3}^{n})_{L^{2}\left(  \Omega
_{2}\right)  }+i\rho_{2}\lambda_{n}\mu_{n}^{-r}(  \widetilde{f}_{3}^{n},w_{3}^{n})_{L^{2}\left(  \Omega_{2}\right)  }\text{.}
\label{ast10}%
\end{align}
Note that from (\ref{ast1}) we have $i\lambda_{n}w_{1}^{n}=\mu_{n}%
^{-r}\widetilde{f}_{1}^{n}+w_{2}^{n}$. Then,%
\begin{equation}
\lambda_{n}^{2}\left\Vert w_{1}^{n}\right\Vert _{L^{2}\left(  \Omega
_{1}\right)  }^{2}=\left\Vert i\lambda_{n}w_{1}^{n}\right\Vert _{L^{2}\left(
\Omega_{1}\right)  }^{2}\leq C\Big(  \mu_{n}^{-2r}\Vert \widetilde{f}_{1}^{n}\Vert _{L^{2}\left(  \Omega_{1}\right)  }^{2}+\left\Vert
w_{2}^{n}\right\Vert _{L^{2}\left(  \Omega_{1}\right)  }^{2}\Big)  \text{.}
\label{ast12}%
\end{equation}

\smallskip
\noindent\textbf{(iii)} \textit{Estimate of $\|w_2^n\|_{H^1(\Omega_1)}$:}
From $D(\mathscr{A})\hookrightarrow H^{4}\left(  \Omega_{1}\right)  \times
H^{2}\left(  \Omega_{1}\right)  \times H^{2}\left(  \Omega_{2}\right)  \times
H^{1}\left(  \Omega_{2}\right)  \times H^{2}\left(  \Omega_{1}\right)  $, it
follows that \textcolor{black}{for some $N\in\mathbb{N}$ it holds}%
\begin{align*}
\left\Vert w^{n}\right\Vert _{H^{4}\left(  \Omega_{1}\right)  \times
H^{2}\left(  \Omega_{1}\right)  \times H^{2}\left(  \Omega_{2}\right)  \times
H^{1}\left(  \Omega_{2}\right)  \times H^{2}\left(  \Omega_{1}\right)  }  &
\leq C\left(  \left\Vert w^{n}\right\Vert _{\mathscr{H}}+\left\Vert
\mathscr{A}w^{n}\right\Vert _{\mathscr{H}}\right) \\
&  \leq C\big(  1+\Vert i\lambda_{n}w^{n}-\mu_{n}^{-r}\widetilde{f}^{n}\Vert _{\mathscr{H}}\big) \\
&  \leq C\big(  1+\mu_{n}+\underbrace{\mu_{n}^{-r}\Vert \widetilde{f}^{n}\Vert _{\mathscr{H}}}_{\leq1\text{, for all }n\geq N}\big) \\
&  \leq C\mu_{n}\text{, \ \ for all }n\geq N\text{,}%
\end{align*}
i.e.,%
\begin{equation}
\mu_{n}^{-1}\left\Vert w^{n}\right\Vert _{H^{4}\left(  \Omega_{1}\right)
\times H^{2}\left(  \Omega_{1}\right)  \times H^{2}\left(  \Omega_{2}\right)
\times H^{1}\left(  \Omega_{2}\right)  \times H^{2}\left(  \Omega_{1}\right)
}\leq C\text{, \ \ for all }n\geq N\text{.} \label{ast13}%
\end{equation}
Therefore,  both sequences 
\begin{equation}
\left(  \mu_{n}^{-1}\left\Vert w_{1}^{n}\right\Vert _{H^{4}\left(  \Omega
_{1}\right)  }\right)_{n\in\N}\text{ and }\left(  \mu_{n}^{-1}\left\Vert
w_{2}^{n}\right\Vert _{H^{2}\left(  \Omega_{1}\right)  }\right)_{n\in\N} \label{ast14}%
\end{equation}
are bounded. Now, by Sobolev's interpolation inequality, we get%
$$
\mu_{n}^{-1/2}\left\Vert w_{2}^{n}\right\Vert _{H^{1}\left(  \Omega
_{1}\right)  }\leq C\underbrace{\left\Vert w_{2}^{n}\right\Vert _{L^{2}\left(
\Omega_{1}\right)  }^{1/2}}_{\leq C}\underbrace{\mu_{n}^{-1/2}\left\Vert
w_{2}^{n}\right\Vert _{H^{2}\left(  \Omega_{1}\right)  }^{1/2}}_{\leq C}\leq
C\text{,}%
$$
i.e.,%
\begin{equation}
\mu_{n}^{-1/2}\left\Vert w_{2}^{n}\right\Vert _{H^{1}\left(  \Omega
_{1}\right)  }\leq C\text{ \ \ for all }n\in\mathbb{N}\text{.} \label{ast15}%
\end{equation}

\smallskip
\noindent\textbf{(iv)} \textit{Estimate of $\|\Delta w_1^n\|_{L^2(\Omega_1)}$:}
Note that, due to (\ref{ast5}),%
$$
\mu_{n}^{\varepsilon r-\delta}\left(  i\rho_{0}\lambda_{n}w_{5}^{n}%
-\alpha\Delta w_{2}^{n}-\beta\Delta w_{5}^{n}+\sigma w_{5}^{n}\right)
=\textcolor{black}{\mu_n^{(\varepsilon-1)r-\delta}}\widetilde{f}_{5}^{n}\rightarrow0\text{ in }L^{2}\left(  \Omega_{1}\right)
$$
for any $0\leq\varepsilon,\delta\leq1$. From this and (\ref{estimate dissip})
we deduce%
$$
\mu_{n}^{\varepsilon r-\delta}\left(  \alpha\Delta w_{2}^{n}+\beta\Delta
w_{5}^{n}\right)  \rightarrow0\text{ in }L^{2}\left(  \Omega_{1}\right)
\text{,}%
$$
whenever $\varepsilon r-\delta+1\leq r/2.$ As $\mu_{n}^{\varepsilon r-\delta}%
w_{2}^{n}=i\mu_{n}^{\varepsilon r-\delta}\lambda_{n}w_{1}^{n}-\mu
_{n}^{-\left(  \left(  1-\varepsilon\right)  r+\delta\right)  }\widetilde{f}%
_{1}^{n}$ and $\mu_{n}^{-\left(  \left(  1-\varepsilon\right)  r+\delta
\right)  }\Delta\widetilde{f}_{1}^{n}\rightarrow0$ in $L^{2}\left(  \Omega
_{1}\right)  $, then%
\begin{equation}
i\alpha\mu_{n}^{\varepsilon r-\delta}\lambda_{n}\Delta w_{1}^{n}+\beta\mu
_{n}^{\varepsilon r-\delta}\Delta w_{5}^{n}\rightarrow0\text{ in }L^{2}\left(
\Omega_{1}\right)  \text{,} \label{ast15b}%
\end{equation}
whenever $\varepsilon r-\delta+1\leq r/2$. Now, let $\varepsilon r-\delta+1\leq
r/2$ with $0\leq\varepsilon,\delta\leq1$. From (\ref{ast15b}) and integration
by parts it follows
\begin{multline}\label{ast15c}
 \big(  i\alpha\mu_{n}^{\varepsilon r-\delta}\lambda_{n}\Delta w_{1}%
^{n}+\beta\mu_{n}^{\varepsilon r-\delta}\Delta w_{5}^{n},\Delta w_{1}^{n}\big)_{L^{2}\left(  \Omega_{1}\right)  } = i\alpha\mu_{n}^{\varepsilon r-\delta}\lambda_{n}\left\Vert \Delta
w_{1}^{n}\right\Vert _{L^{2}\left(  \Omega_{1}\right)  }^{2}\\
- \beta\mu_{n}^{\varepsilon r-\delta}\left(  \nabla w_{5}^{n},\nabla\Delta w_{1}%
^{n}\right)  _{L^{2}\left(  \Omega_{1}\right)  ^{2}}+\beta\mu_{n}^{\varepsilon
r-\delta}\left(  \partial_{\nu}w_{5}^{n},\Delta w_{1}^{n}\right)
_{L^{2}\left(  \partial\Omega_{1}\right)  }\text{,} %
\end{multline}
where%
\begin{align}
\big\vert \mu_{n}^{\varepsilon r-\delta}&\left(  \nabla w_{5}^{n}, \nabla\Delta
w_{1}^{n}\right)  _{L^{2}\left(  \Omega_{1}\right)  ^{2}}\big\vert \nonumber\\
& \leq C\mu_{n}^{\varepsilon r-\delta}\left\Vert \nabla w_{5}^{n}\right\Vert
_{L^{2}\left(  \Omega_{1}\right)  ^{2}}\left\Vert w_{1}^{n}\right\Vert
_{H^{4}\left(  \Omega_{1}\right)  }\nonumber\\
&  \leq C\mu_{n}^{\varepsilon r-\delta}\mu_{n}^{-r/2}\left(  \mu_{n}%
^{r/2}\left\Vert \nabla w_{5}^{n}\right\Vert _{L^{2}\left(  \Omega_{1}\right)
^{2}}\right)  \mu_{n}\left(  \mu_{n}^{-1}\left\Vert w_{1}^{n}\right\Vert
_{H^{4}\left(  \Omega_{1}\right)  }\right) \nonumber\\
&  \leq C\mu_{n}^{\varepsilon r-\delta-r/2+1}\left(  \mu_{n}^{r/2}\left\Vert
\nabla w_{5}^{n}\right\Vert _{L^{2}\left(  \Omega_{1}\right)  ^{2}}\right)
\rightarrow0\text{,} \label{ast15d2}%
\end{align}%
\begin{align}
\big\vert \mu_{n}^{\varepsilon r-\delta}&\left(  \partial_{\nu}w_{5}%
^{n},\Delta w_{1}^{n}\right)  _{L^{2}\left(  \partial\Omega_{1}\right)
}\big\vert  \nonumber\\
 & \leq C\mu_{n}^{\varepsilon r-\delta}\left\Vert \partial_{\nu
}w_{5}^{n}\right\Vert _{L^{2}\left(  \partial\Omega_{1}\right)  }\left\Vert
\Delta w_{1}^{n}\right\Vert _{L^{2}\left(  \partial\Omega_{1}\right)
}\nonumber\\
&  \leq C\mu_{n}^{\varepsilon r-\delta}\kappa\left\Vert w_{5}^{n}\right\Vert
_{L^{2}\left(  \partial\Omega_{1}\right)  }\mu_{n}\left(  \mu_{n}%
^{-1}\left\Vert w_{1}^{n}\right\Vert _{H^{4}\left(  \Omega_{1}\right)
}\right) \nonumber\\
&  \leq C\mu_{n}^{\varepsilon r-\delta+1}\left\Vert \nabla w_{5}%
^{n}\right\Vert _{L^{2}\left(  \Omega_{1}\right)  ^{2}}\rightarrow0\text{.}
\label{ast15e}%
\end{align}
So, it follows from (\ref{ast15b})-(\ref{ast15e}) that%
\begin{equation}
\mu_{n}^{\varepsilon r-\delta+1}\left\Vert \Delta w_{1}^{n}\right\Vert
_{L^{2}\left(  \Omega_{1}\right)  }^{2}\rightarrow0\text{,} \label{ast16}%
\end{equation}
whenever $\varepsilon r-\delta+1\leq r/2$ and $0\leq\varepsilon,\delta\leq1$,
and therefore%
\begin{equation}
\mu_{n}^{\frac{\varepsilon r-\delta+1}{2}}\left\Vert \Delta w_{1}%
^{n}\right\Vert _{L^{2}\left(  \Omega_{1}\right)  }\rightarrow0\text{,}
\label{ast16b}%
\end{equation}
whenever $\varepsilon r-\delta+1\leq r/2$ and $0\leq\varepsilon,\delta\leq1$. 

\smallskip
\noindent\textbf{(v)} \textit{Estimate of $\|w_2^n\|_{L^2(\Omega_1)}$:}
Multiplying (\ref{ast2}) by $\alpha\overline{w_{2}^{n}}$ in $L^{2}\left(
\Omega_{1}\right)  $, we get%
\begin{multline*}
\mu_{n}^{r}\big[  i\alpha\rho_{1}\lambda_{n}\left\Vert w_{2}^{n}\right\Vert
_{L^{2}\left(  \Omega_{1}\right)  }^{2}+\alpha\beta_{1}\left(  \Delta^{2}%
w_{1}^{n},w_{2}^{n}\right)  _{L^{2}\left(  \Omega_{1}\right)  }+\alpha
^{2}\left(  \Delta w_{5}^{n},w_{2}^{n}\right)  _{L^{2}\left(  \Omega
_{1}\right)  }\big] \\
 = \alpha(  \widetilde{f}_{2}^{n},w_{2}^{n})_{L^{2}\left(  \Omega_{1}\right)  }\text{.}%
\end{multline*}
Then,%
\begin{multline}
i\alpha\rho_{1}\mu_{n}^{r}\lambda_{n}\left\Vert w_{2}^{n}\right\Vert
_{L^{2}\left(  \Omega_{1}\right)  }^{2}=\alpha\underbrace{(
\widetilde{f}_{2}^{n},w_{2}^{n})_{L^{2}\left(  \Omega_{1}\right)  }%
}_{=:I_{1}}-\beta_{1}\underbrace{\alpha\mu_{n}^{r}\left(  \Delta^{2}w_{1}%
^{n},w_{2}^{n}\right)  _{L^{2}\left(  \Omega_{1}\right)  }}_{=:I_{2}}\\
-\alpha^{2}\underbrace{\mu_{n}^{r}\left(  \Delta w_{5}^{n},w_{2}^{n}\right)
_{L^{2}\left(  \Omega_{1}\right)  }}_{=:I_{3}}\text{.} \label{ast17}%
\end{multline}
\textbf{Estimate for $I_{1}$}:%
\begin{equation}
\left\vert I_{1}\right\vert \leq \Vert \widetilde{f}_{2}^{n}\Vert
_{L^{2}\left(  \Omega_{1}\right)  }\left\Vert w_{2}^{n}\right\Vert
_{L^{2}\left(  \Omega_{1}\right)  }\leq C \Vert \widetilde{f}_{2}^{n}\Vert_{L^{2}\left(  \Omega_{1}\right)  }\rightarrow0\text{.}
\label{ast17b}%
\end{equation}
\textbf{Estimate for $I_{3}$}:%
\begin{align*}
\left\vert I_{3}\right\vert  &  =\mu_{n}^{r}\big\vert -\left(  \nabla
w_{5}^{n},\nabla w_{2}^{n}\right)  _{L^{2}\left(  \Omega_{1}\right)  ^{2}}+\left(  \partial_{\nu}w_{5}^{n},w_{2}^{n}\right)  _{L^{2}\left(
\partial\Omega_{1}\right)  }\big\vert \\
&  =\mu_{n}^{r}\big\vert -\left(  \nabla w_{5}^{n},\nabla w_{2}^{n}\right)
_{L^{2}\left(  \Omega_{1}\right)  ^{2}}-\kappa\left(  w_{5}^{n},w_{2}^{n}\right)_{L^{2}\left(  I\right)  }\big\vert \\
&  \leq\mu_{n}^{r}\left\Vert \nabla w_{5}^{n}\right\Vert _{L^{2}\left(
\Omega_{1}\right)  ^{2}}\left\Vert \nabla w_{2}^{n}\right\Vert _{L^{2}\left(
\Omega_{1}\right)  ^{2}}+\kappa\mu_{n}^{r}\left\Vert w_{5}^{n}\right\Vert
_{L^{2}(I)}\left\Vert w_{2}^{n}\right\Vert _{L^{2}(I)}\\
&  \!\!\!\!\!\!\!\!\!\!\!\!\overset{\text{Trace theorem}}{\leq}C\mu_{n}^{r}\left\Vert w_{5}%
^{n}\right\Vert _{H^{1}(\Omega_{1})}\left\Vert w_{2}^{n}\right\Vert
_{H^{1}(\Omega_{1})}\\
&  \leq C\mu_{n}^{r}\mu_{n}^{-r/2}\left(  \mu_{n}^{r/2}\left\Vert w_{5}%
^{n}\right\Vert _{H^{1}(\Omega_{1})}\right)  \mu_{n}^{1/2}\left(  \mu
_{n}^{-1/2}\left\Vert w_{2}^{n}\right\Vert _{H^{1}(\Omega_{1})}\right) \\
&  \leq C\mu_{n}^{r/2}\mu_{n}^{1/2}\left(  \mu_{n}^{r/2}\left\Vert w_{5}%
^{n}\right\Vert _{H^{1}(\Omega_{1})}\right)  \text{.}%
\end{align*}
Then%
\begin{equation}
\frac{\left\vert I_{3}\right\vert }{\mu_{n}^{r/2}\mu_{n}^{1/2}}\leq C\left(
\mu_{n}^{r/2}\left\Vert w_{5}^{n}\right\Vert _{H^{1}(\Omega_{1})}\right)
\rightarrow0\text{.} \label{ast18}%
\end{equation}
\textbf{Estimate for $I_{2}$}: By integration by parts and the equality
$$
\alpha\mu_{n}^{r}\Delta w_{2}^{n}=i\rho_{0}\mu_{n}^{r}\lambda_{n}w_{5}%
^{n}-\beta\mu_{n}^{r}\Delta w_{5}^{n}+\sigma\mu_{n}^{r}w_{5}^{n}%
-\widetilde{f}_{5}^{n}%
$$
(see (\ref{ast5})) we have%
\begin{align}
I_{2}  &  =\alpha\mu_{n}^{r}\big(  \Delta(  \underbrace{\Delta w_{1}^{n}}_{=:y})  ,w_{2}^{n}\big)_{L^{2}\left(  \Omega_{1}\right)
}\nonumber\\
&  =\alpha\mu_{n}^{r}\left[  \left(  y,\Delta w_{2}^{n}\right)  _{L^{2}\left(
\Omega_{1}\right)  }+\left(  \partial_{\nu}y,w_{2}^{n}\right)  _{L^{2}\left(
\partial\Omega_{1}\right)  }-\left(  y,\partial_{\nu}w_{2}^{n}\right)
_{L^{2}\left(  \partial\Omega_{1}\right)  }\right] \nonumber\\
&  =\left(  \Delta w_{1}^{n},\alpha\mu_{n}^{r}\Delta w_{2}^{n}\right)
_{L^{2}\left(  \Omega_{1}\right)  }+\alpha\mu_{n}^{r}\left(  \partial_{\nu
}\Delta w_{1}^{n},w_{2}^{n}\right)  _{L^{2}\left(  \partial\Omega_{1}\right)
}\nonumber\\
& \qquad -\alpha\mu_{n}^{r} \left(  \Delta w_{1}^{n},\partial_{\nu}w_{2}^{n}\right)
_{L^{2}\left(  \partial\Omega_{1}\right)  }\nonumber\\
&  = (  \Delta w_{1}^{n},i\rho_{0}\mu_{n}^{r}\lambda_{n}w_{5}^{n}-\beta
\mu_{n}^{r}\Delta w_{5}^{n}+\sigma\mu_{n}^{r}w_{5}^{n}-\widetilde{f}_{5}^{n})_{L^{2}\left(  \Omega_{1}\right)  }\nonumber\\
&  \qquad+\alpha\mu_{n}^{r}\left(  \partial_{\nu}\Delta w_{1}^{n},w_{2}%
^{n}\right)  _{L^{2}\left(  \partial\Omega_{1}\right)  }-\alpha\mu_{n}%
^{r}\left(  \Delta w_{1}^{n},\partial_{\nu}w_{2}^{n}\right)  _{L^{2}\left(
\partial\Omega_{1}\right)  }\nonumber\\
&  =-i\rho_{0}\mu_{n}^{r}\lambda_{n}\left(  \Delta w_{1}^{n},w_{5}^{n}\right)
_{L^{2}\left(  \Omega_{1}\right)  }-\beta\mu_{n}^{r}\left(  \Delta w_{1}%
^{n},\Delta w_{5}^{n}\right)  _{L^{2}\left(  \Omega_{1}\right)  }\nonumber\\
& \qquad +\sigma\mu_{n}^{r}\left(  \Delta w_{1}^{n},w_{5}^{n}\right)  _{L^{2}\left(
\Omega_{1}\right)  }- (  \Delta w_{1}^{n},\widetilde{f}_{5}^{n})_{L^{2}\left(  \Omega_{1}\right)  }\nonumber\\
&  \qquad+\alpha\mu_{n}^{r}\left(  \partial_{\nu}\Delta w_{1}^{n},w_{2}%
^{n}\right)  _{L^{2}\left(  \partial\Omega_{1}\right)  }-\alpha\mu_{n}%
^{r}\left(  \Delta w_{1}^{n},\partial_{\nu}w_{2}^{n}\right)  _{L^{2}\left(
\partial\Omega_{1}\right)  }\text{.} \label{ast19}%
\end{align}
Now, we will estimate each term on the right side of the equality
(\ref{ast19}).%
\begin{align*}
\big\vert i\rho_{0}&\mu_{n}^{r}\lambda_{n}\left(  \Delta w_{1}^{n},w_{5}%
^{n}\right)  _{L^{2}\left(  \Omega_{1}\right)  }\big\vert  \\
& \leq C\mu
_{n}^{r}\mu_{n}\left\Vert \Delta w_{1}^{n}\right\Vert _{L^{2}(\Omega_{1}%
)}\left\Vert w_{5}^{n}\right\Vert _{L^{2}(\Omega_{1})}\\
&  \leq C\mu_{n}^{r/2}\mu_{n}\Big(  \mu_{n}^{r/2}\left\Vert w_{5}%
^{n}\right\Vert _{L^{2}(\Omega_{1})}\Big)  \mu_{n}^{-\frac{\varepsilon
r-\delta+1}{2}} \Big(  \mu_{n}^{\frac{\varepsilon r-\delta+1}{2}}\left\Vert
\Delta w_{1}^{n}\right\Vert _{L^{2}(\Omega_{1})}\Big) \\
&  =C\mu_{n}^{\frac{\left(  1-\varepsilon\right)  r+\delta-1}{2}}\mu
_{n}\left(  \mu_{n}^{r/2}\left\Vert w_{5}^{n}\right\Vert _{L^{2}(\Omega_{1}%
)}\right)  \Big(  \mu_{n}^{\frac{\varepsilon r-\delta+1}{2}}\left\Vert \Delta
w_{1}^{n}\right\Vert _{L^{2}(\Omega_{1})}\Big)  \text{.}%
\end{align*}
Then,%
\begin{equation}
\frac{\big\vert i\rho_{0}\mu_{n}^{r}\lambda_{n}\left(  \Delta w_{1}^{n}%
,w_{5}^{n}\right)  _{L^{2}\left(  \Omega_{1}\right)  }\big\vert }{\mu
_{n}^{\frac{\left(  1-\varepsilon\right)  r+\delta-1}{2}}\mu_{n}}%
\rightarrow0\text{,} \label{ast19-1}%
\end{equation}
due to (\ref{estimate dissip}) and(\ref{ast16b}). From Lemma
\ref{boundary}, Trace theorem and Sobolev's interpolation inequality, it
follows%
\begin{align*}
&  \big\vert \beta\mu_{n}^{r}\left(  \Delta w_{1}^{n},\Delta w_{5}^{n}\right)_{L^{2}\left(  \Omega_{1}\right)  }\big\vert \\
&  = \beta\mu_{n}^{r}\big\vert -\left(  \nabla\Delta w_{1}^{n},\nabla
w_{5}^{n}\right)_{L^{2}\left(  \Omega_{1}\right)^{2}} + (  \Delta
w_{1}^{n},\underbrace{\partial_{\nu}w_{5}^{n}}_{=-\kappa w_{5}^{n}})
_{L^{2}\left(  \partial\Omega_{1}\right)  }\big\vert \\
&  \leq C\mu_{n}^{r}\left(  \left\Vert \nabla\Delta w_{1}^{n}\right\Vert
_{L^{2}(\Omega_{1})^{2}}\left\Vert \nabla w_{5}^{n}\right\Vert _{L^{2}%
(\Omega_{1})^{2}}+\left\Vert \Delta w_{1}^{n}\right\Vert _{L^{2}\left(
\partial\Omega_{1}\right)  }\left\Vert w_{5}^{n}\right\Vert _{L^{2}\left(
\partial\Omega_{1}\right)  }\right) \\
&  \leq C\mu_{n}^{r}\left(  \left\Vert w_{1}^{n}\right\Vert _{H^{4}(\Omega
_{1})}\left\Vert w_{5}^{n}\right\Vert _{H^{1}(\Omega_{1})}+\left\Vert \Delta
w_{1}^{n}\right\Vert _{L^{2}\left(  \Omega_{1}\right)  }^{1/2}\left\Vert
\Delta w_{1}^{n}\right\Vert _{H^{1}\left(  \Omega_{1}\right)  }^{1/2}%
\left\Vert w_{5}^{n}\right\Vert _{H^{1}\left(  \Omega_{1}\right)  }\right) \\
&  \leq C\mu_{n}^{r}\Big(  \left\Vert w_{1}^{n}\right\Vert _{H^{4}(\Omega
_{1})}\left\Vert w_{5}^{n}\right\Vert _{H^{1}(\Omega_{1})}\\
& \qquad\qquad +\left\Vert \Delta w_{1}^{n}\right\Vert _{L^{2}\left(  \Omega_{1}\right)  }^{1/2}\left(
\left\Vert \Delta w_{1}^{n}\right\Vert _{L^{2}\left(  \Omega_{1}\right)
}^{1/2}\left\Vert \Delta w_{1}^{n}\right\Vert _{H^{2}\left(  \Omega
_{1}\right)  }^{1/2}\right)  ^{1/2}\left\Vert w_{5}^{n}\right\Vert
_{H^{1}\left(  \Omega_{1}\right)  }\Big) \\
&  \leq C\mu_{n}^{r}\left(  \left\Vert w_{1}^{n}\right\Vert _{H^{4}(\Omega
_{1})}\left\Vert w_{5}^{n}\right\Vert _{H^{1}(\Omega_{1})}+\left\Vert \Delta
w_{1}^{n}\right\Vert _{L^{2}\left(  \Omega_{1}\right)  }^{3/4}\left\Vert
w_{1}^{n}\right\Vert _{H^{4}\left(  \Omega_{1}\right)  }^{1/4}\left\Vert
w_{5}^{n}\right\Vert _{H^{1}\left(  \Omega_{1}\right)  }\right) \\
&  = C\Bigg(  \mu_{n}^{r/2+1}\frac{\left\Vert w_{1}^{n}\right\Vert
_{H^{4}(\Omega_{1})}}{\mu_{n}}  \\
& \qquad  + \mu_{n}^{\frac{r}{2}-\frac{3\left(  \varepsilon
r-\delta+1\right)  }{8}+\frac{1}{4}}\bigg(  \mu_{n}^{\frac{\varepsilon
r-\delta+1}{2}}\left\Vert \Delta w_{1}^{n}\right\Vert _{L^{2}\left(
\Omega_{1}\right)  }\bigg)^{3/4}\bigg(  \frac{\left\Vert w_{1}^{n}\right\Vert _{H^{4}\left(  \Omega_{1}\right)  }}{\mu_{n}}\bigg)^{1/4}\Bigg)\\
& \qquad\qquad \cdot \mu_{n}^{r/2}\left\Vert w_{5}^{n}\right\Vert _{H^{1}(\Omega
_{1})}\\
&  \leq C\mu_{n}^{r/2+1}\Big(  1+\mu_{n}^{-3\left(  \varepsilon
r-\delta+3\right)  /8}\Big(  \mu_{n}^{\frac{\varepsilon r-\delta+1}{2}%
}\left\Vert \Delta w_{1}^{n}\right\Vert _{L^{2}\left(  \Omega_{1}\right)
}\Big)^{3/4}\Big)  \mu_{n}^{r/2}\left\Vert w_{5}^{n}\right\Vert
_{H^{1}(\Omega_{1})}\text{,}%
\end{align*}
\textcolor{black}{where in the last inequality it was used that}  $\left(  \mu_{n}^{-1}\left\Vert w_{1}^{n}\right\Vert _{H^{4}\left(  \Omega_{1}\right)  }\right)_{n\in\N}$ is bounded.
Then,
\begin{equation}
\frac{\left\vert \beta\mu_{n}^{r}\left(  \Delta w_{1}^{n},\Delta w_{5}%
^{n}\right)  _{L^{2}\left(  \Omega_{1}\right)  }\right\vert }{\mu_{n}^{r/2+1}%
}\rightarrow0 \label{ast19-2}%
\end{equation}
due to (\ref{estimate dissip}) and(\ref{ast16b}). Moreover,
$$
\left\vert \sigma\mu_{n}^{r}\left(  \Delta w_{1}^{n},w_{5}^{n}\right)
_{L^{2}\left(  \Omega_{1}\right)  }\right\vert \leq C\mu_{n}^{\frac{(1-\varepsilon)r + \delta -1}{2}}\mu
_{n}^{\frac{\varepsilon r - \delta +1}{2}}\left\Vert \Delta w_{1}^{n}\right\Vert _{L^{2}\left(  \Omega
_{1}\right)  }\mu_{n}^{r/2}\left\Vert w_{5}^{n}\right\Vert _{L^{2}\left(
\Omega_{1}\right)  }\text{.}%
$$
Then,%
\begin{equation}
\frac{\left\vert \sigma\mu_{n}^{r}\left(  \Delta w_{1}^{n},w_{5}^{n}\right)
_{L^{2}\left(  \Omega_{1}\right)  }\right\vert }{\mu_{n}^{\frac{(1-\varepsilon)r + \delta -1}{2}}}\rightarrow
0\text{.} \label{ast19-3}%
\end{equation}
Furthermore, from $\left\Vert \Delta w_{1}^{n}\right\Vert _{L^{2}\left(
\Omega_{1}\right)  }\leq\left\Vert w_{1}^{n}\right\Vert _{H^{2}\left(
\Omega_{1}\right)  }\leq C\left\Vert w^{n}\right\Vert _{\mathscr{H}}\leq C$ it
follows that%
$$
\vert (  \Delta w_{1}^{n},\widetilde{f}_{5}^{n})_{L^{2}\left(  \Omega_{1}\right)  }\vert \leq\left\Vert \Delta w_{1}^{n}\right\Vert _{L^{2}\left(  \Omega_{1}\right)  }\Vert \widetilde{f}_{5}^{n}\Vert _{L^{2}\left(  \Omega_{1}\right)  }\leq C\Vert
\widetilde{f}_{5}^{n}\Vert _{L^{2}\left(  \Omega_{1}\right)
}\rightarrow0\text{,}%
$$
i.e.,%
\begin{equation}
\vert (  \Delta w_{1}^{n},\widetilde{f}_{5}^{n})_{L^{2}\left(  \Omega_{1}\right)  }\vert \rightarrow0\text{.}
\label{ast19-4}%
\end{equation}
Now, From Lemma \ref{boundary}, Corollary \ref{coro-na}, Sobolev's
interpolation inequality, (\ref{ast14}) and (\ref{ast15}) it follows%
\begin{align*}
&  \left\vert \alpha\mu_{n}^{r}\left(  \partial_{\nu}\Delta w_{1}^{n}%
,w_{2}^{n}\right)  _{L^{2}\left(  \partial\Omega_{1}\right)  }\right\vert \\
&  \leq C\mu_{n}^{r}\left\Vert \partial_{\nu}\Delta w_{1}^{n}\right\Vert
_{L^{2}\left(  I\right)  }\left\Vert w_{2}^{n}\right\Vert _{L^{2}\left(
I\right)  }\\
&  \leq C\mu_{n}^{r}\left\Vert \Delta w_{1}^{n}\right\Vert _{H^{1}\left(
\Omega_{1}\right)  }^{1/2}\left\Vert \Delta w_{1}^{n}\right\Vert
_{H^{2}\left(  \Omega_{1}\right)  }^{1/2}\underbrace{\left\Vert w_{2}%
^{n}\right\Vert _{L^{2}\left(  \Omega_{1}\right)  }^{1/2}}_{\leq C}\left\Vert
w_{2}^{n}\right\Vert _{H^{1}\left(  \Omega_{1}\right)  }^{1/2}\\
&  \leq C\mu_{n}^{r}\left(  \left\Vert \Delta w_{1}^{n}\right\Vert
_{L^{2}\left(  \Omega_{1}\right)  }^{1/2}\left\Vert \Delta w_{1}%
^{n}\right\Vert _{H^{2}\left(  \Omega_{1}\right)  }^{1/2}\right)
^{1/2}\left\Vert \Delta w_{1}^{n}\right\Vert _{H^{2}\left(  \Omega_{1}\right)
}^{1/2}\left\Vert w_{2}^{n}\right\Vert _{H^{1}\left(  \Omega_{1}\right)
}^{1/2}\\
&  \leq C\mu_{n}^{r}\left\Vert \Delta w_{1}^{n}\right\Vert _{L^{2}\left(
\Omega_{1}\right)  }^{1/4}\left\Vert w_{1}^{n}\right\Vert _{H^{4}\left(
\Omega_{1}\right)  }^{3/4}\left\Vert w_{2}^{n}\right\Vert _{H^{1}\left(
\Omega_{1}\right)  }^{1/2}\\
&  =C\mu_{n}^{r}\mu_{n}^{-\frac{\varepsilon r-\delta+1}{2}\cdot\frac{1}{4}%
}\left(  \mu_{n}^{\frac{\varepsilon r-\delta+1}{2}}\left\Vert \Delta w_{1}%
^{n}\right\Vert _{L^{2}(\Omega_{1})}\right)  ^{1/4}\mu_{n}^{3/4}\left(
\mu_{n}^{-1}\left\Vert w_{1}^{n}\right\Vert _{H^{4}\left(  \Omega_{1}\right)
}\right)^{3/4}\\
& \qquad \cdot \mu_{n}^{1/4}\left(  \mu_{n}^{-1/2}\left\Vert w_{2}%
^{n}\right\Vert _{H^{1}\left(  \Omega_{1}\right)  }\right)^{1/2}\\
&  \leq C\mu_{n}^{r+1-\frac{\varepsilon r-\delta+1}{8}}\left(  \mu_{n}%
^{\frac{\varepsilon r-\delta+1}{2}}\left\Vert \Delta w_{1}^{n}\right\Vert
_{L^{2}\left(  \Omega_{1}\right)  }\right)  ^{1/4}\text{.}%
\end{align*}
Then,%
\begin{equation}
\frac{\left\vert \alpha\mu_{n}^{r}\left(  \partial_{\nu}\Delta w_{1}^{n}%
,w_{2}^{n}\right)  _{L^{2}\left(  \partial\Omega_{1}\right)  }\right\vert
}{\mu_{n}^{r+1-\frac{\varepsilon r-\delta+1}{8}}}\rightarrow0\text{,}
\label{ast19-5}%
\end{equation}
due to (\ref{ast16b}). Again, from Lemma \ref{boundary}, Corollary
\ref{coro-na}, Sobolev's interpolation inequality, (\ref{ast14}) and
(\ref{ast15}) it follows%
\begin{align*}
&  \left\vert \alpha\mu_{n}^{r}\left(  \Delta w_{1}^{n},\partial_{\nu}%
w_{2}^{n}\right)  _{L^{2}\left(  \partial\Omega_{1}\right)  }\right\vert \\
&  \leq C\mu_{n}^{r}\left\Vert \Delta w_{1}^{n}\right\Vert _{L^{2}\left(
I\right)  }\left\Vert \partial_{\nu}w_{2}^{n}\right\Vert _{L^{2}\left(
I\right)  }\\
&  \leq C\mu_{n}^{r}\left\Vert \Delta w_{1}^{n}\right\Vert _{L^{2}\left(
\Omega_{1}\right)  }^{1/2}\left\Vert \Delta w_{1}^{n}\right\Vert
_{H^{1}\left(  \Omega_{1}\right)  }^{1/2}\left\Vert w_{2}^{n}\right\Vert
_{H^{1}\left(  \Omega_{1}\right)  }^{1/2}\left\Vert w_{2}^{n}\right\Vert
_{H^{2}\left(  \Omega_{1}\right)  }^{1/2}\\
&  \leq C\mu_{n}^{r}\left\Vert \Delta w_{1}^{n}\right\Vert _{L^{2}\left(
\Omega_{1}\right)  }^{1/2}\left(  \left\Vert \Delta w_{1}^{n}\right\Vert
_{L^{2}\left(  \Omega_{1}\right)  }^{1/2}\left\Vert \Delta w_{1}%
^{n}\right\Vert _{H^{2}\left(  \Omega_{1}\right)  }^{1/2}\right)
^{1/2}\left\Vert w_{2}^{n}\right\Vert _{H^{1}\left(  \Omega_{1}\right)
}^{1/2}\left\Vert w_{2}^{n}\right\Vert _{H^{2}\left(  \Omega_{1}\right)
}^{1/2}\\
&  \leq C\mu_{n}^{r}\left\Vert \Delta w_{1}^{n}\right\Vert _{L^{2}\left(
\Omega_{1}\right)  }^{3/4}\left\Vert w_{1}^{n}\right\Vert _{H^{4}\left(
\Omega_{1}\right)  }^{1/4}\left\Vert w_{2}^{n}\right\Vert _{H^{1}\left(
\Omega_{1}\right)  }^{1/2}\left\Vert w_{2}^{n}\right\Vert _{H^{2}\left(
\Omega_{1}\right)  }^{1/2}\\
&  =C\mu_{n}^{r}\mu_{n}^{-\frac{\varepsilon r-\delta+1}{2}\cdot\frac{3}{4}%
}\left(  \mu_{n}^{\frac{\varepsilon r-\delta+1}{2}}\left\Vert \Delta w_{1}%
^{n}\right\Vert _{L^{2}\left(  \Omega_{1}\right)  }\right)  ^{3/4}\mu
_{n}^{1/4}\left(  \mu_{n}^{-1}\left\Vert w_{1}^{n}\right\Vert _{H^{4}\left(
\Omega_{1}\right)  }\right)^{1/4} \\
& \qquad \cdot \mu_{n}^{1/4}\left(  \mu_{n}^{-1/2}%
\left\Vert w_{2}^{n}\right\Vert _{H^{1}\left(  \Omega_{1}\right)  }\right)
^{1/2}\mu_{n}^{1/2}\left(  \mu_{n}^{-1}\left\Vert w_{2}^{n}\right\Vert
_{H^{2}\left(  \Omega_{1}\right)  }\right)^{1/2}\\
&  \leq C\mu_{n}^{r+1-\frac{3\left(  \varepsilon r-\delta+1\right)  }{8}%
}\left(  \mu_{n}^{\frac{\varepsilon r-\delta+1}{2}}\left\Vert \Delta w_{1}%
^{n}\right\Vert _{L^{2}\left(  \Omega_{1}\right)  }\right)  ^{3/4}\text{.}%
\end{align*}
Then,%
\begin{equation}
\frac{\left\vert \alpha\mu_{n}^{r}\left(  \Delta w_{1}^{n},\partial_{\nu}%
w_{2}^{n}\right)  _{L^{2}\left(  \partial\Omega_{1}\right)  }\right\vert }%
{\mu_{n}^{r+1-\frac{3\left(  \varepsilon r-\delta+1\right)  }{8}}}%
\rightarrow0\text{,} \label{ast19-6}%
\end{equation}
due to (\ref{ast16b}).

Note that the largest denominator between (\ref{ast18}), (\ref{ast19-1}),
(\ref{ast19-2}), (\ref{ast19-3}), (\ref{ast19-5}), and (\ref{ast19-6}) is
$\mu_{n}^{r+1-\frac{\varepsilon r-\delta+1}{8}}$. Therefore, it follows from
(\ref{ast17})-(\ref{ast19-6}) that%
$$
\mu_{n}^{\frac{\varepsilon r-\delta+1}{8}}\left\Vert w_{2}^{n}\right\Vert
_{L^{2}\left(  \Omega_{1}\right)  }^{2}=\frac{\alpha\rho_{1}\mu_{n}^{r}\mu
_{n}\left\Vert w_{2}^{n}\right\Vert _{L^{2}\left(  \Omega_{1}\right)  }^{2}%
}{\alpha\rho_{1}\mu_{n}^{r+1-\frac{\varepsilon r-\delta+1}{8}}}\leq C\left(
\frac{\left\vert I_{1}\right\vert +\left\vert I_{2}\right\vert +\left\vert
I_{3}\right\vert }{\mu_{n}^{r+1-\frac{\varepsilon r-\delta+1}{8}}}\right)
\rightarrow0\text{,}%
$$
whenever $\varepsilon r-\delta+1\leq r/2$ with $0\leq\varepsilon,\delta\leq1$,
i.e.,%
\begin{equation}
\mu_{n}^{\frac{\varepsilon r-\delta+1}{16}}\left\Vert w_{2}^{n}\right\Vert
_{L^{2}\left(  \Omega_{1}\right)  }\rightarrow0\text{,} \label{ast20}%
\end{equation}
whenever $\varepsilon r-\delta+1\leq r/2$ with $0\leq\varepsilon,\delta\leq1$.

\smallskip
\noindent\textbf{(vi)} \textit{Estimate of $\|w_1^n\|_{L^2(I)}$:}
Furthermore, from (\ref{ast12}) and (\ref{ast20}) we obtain%
$$
\mu_{n}^{\frac{\varepsilon r-\delta+1}{8}}\lambda_{n}^{2}\left\Vert w_{1}%
^{n}\right\Vert _{L^{2}\left(  \Omega_{1}\right)  }^{2}\leq\left(  \mu
_{n}^{-2r+\frac{\varepsilon r-\delta+1}{8}}\left\Vert \widetilde{f}_{1}%
^{n}\right\Vert _{L^{2}\left(  \Omega_{1}\right)  }^{2}+\mu_{n}^{\frac
{\varepsilon r-\delta+1}{8}}\left\Vert w_{2}^{n}\right\Vert _{L^{2}\left(
\Omega_{1}\right)  }^{2}\right)  \rightarrow0\text{,}%
$$
i.e.,%
\begin{equation}
\mu_{n}^{\frac{\varepsilon r-\delta+1}{16}}\mu_{n}\left\Vert w_{1}%
^{n}\right\Vert _{L^{2}\left(  \Omega_{1}\right)  }\rightarrow0\text{,}
\label{ast21}%
\end{equation}
whenever $\varepsilon r-\delta+1\leq r/2$ with $0\leq\varepsilon,\delta\leq1$.\\

\textcolor{black}{Note that if $r\geq 2$, $\delta=0$, and $\varepsilon=\frac{1}{2}-\frac{1}{r}\in [0,\frac{1}{2})$, we have $\varepsilon r - \delta + 1 = \frac{r}{2}$ and from \eqref{ast21} it follows
\begin{equation}\label{ast21d}
\Vert w_1^n\Vert_{L^2(\Omega_1)} \to 0.
\end{equation}
}

On the other hand, we obtain by Sobolev's interpolation inequality and the boundedness
of $\left(  \left\Vert w_{1}^{n}\right\Vert _{H^{2}\left(  \Omega_{1}\right)
}\right)_{n\in\N}  $ that%
\begin{equation}
\left\Vert w_{1}^{n}\right\Vert _{H^{1}\left(  \Omega_{1}\right)  }\leq
C\left\Vert w_{1}^{n}\right\Vert _{L^{2}\left(  \Omega_{1}\right)  }%
^{1/2}\left\Vert w_{1}^{n}\right\Vert _{H^{2}\left(  \Omega_{1}\right)
}^{1/2}\leq C\left\Vert w_{1}^{n}\right\Vert _{L^{2}\left(  \Omega_{1}\right)
}^{1/2}\text{.} \label{ast21a}%
\end{equation}
From this and Lemma \ref{boundary} it follows
\begin{multline*}
\left\Vert w_{1}^{n}\right\Vert _{L^{2}\left(  I\right)  }\leq C\left\Vert
w_{1}^{n}\right\Vert _{L^{2}\left(  \Omega_{1}\right)  }^{1/2}\left\Vert
w_{1}^{n}\right\Vert _{H^{1}\left(  \Omega_{1}\right)  }^{1/2}\leq C\left\Vert
w_{1}^{n}\right\Vert _{L^{2}\left(  \Omega_{1}\right)  }^{1/2}\left\Vert
w_{1}^{n}\right\Vert _{L^{2}\left(  \Omega_{1}\right)  }^{1/4}\\
= C\left\Vert w_{1}^{n}\right\Vert _{L^{2}\left(  \Omega_{1}\right)  }^{3/4}\text{,}%
\end{multline*}
i.e.,%
\begin{equation}
\left\Vert w_{1}^{n}\right\Vert _{L^{2}\left(  I\right)  }\leq C\left\Vert
w_{1}^{n}\right\Vert _{L^{2}\left(  \Omega_{1}\right)  }^{3/4}\text{.}
\label{ast21b}%
\end{equation}
Then, due to (\ref{ast21b}) and (\ref{ast21}),
$$
\mu_{n}\left\Vert w_{1}^{n}\right\Vert _{L^{2}\left(  I\right)  }\leq C\mu
_{n}\left(  \mu_{n}^{\frac{\varepsilon r-\delta+1}{16}}\mu_{n}\right)
^{-3/4}\left(  \mu_{n}^{\frac{\varepsilon r-\delta+1}{16}}\mu_{n}\left\Vert
w_{1}^{n}\right\Vert _{L^{2}\left(  \Omega_{1}\right)  }\right)
^{3/4}\rightarrow0\text{,}%
$$
if $\frac{1}{4}-\frac{\varepsilon r-\delta+1}{16}\cdot\frac{3}{4}\leq0$,
i.e.,
\begin{equation}
\mu_{n}\left\Vert w_{1}^{n}\right\Vert _{L^{2}\left(  I\right)  }%
\rightarrow0\text{,} \label{ast21c}%
\end{equation}
if $16/3\leq\varepsilon r-\delta+1\leq r/2$ with $0\leq\varepsilon,\delta
\leq1$. In the following, let $16/3\leq\varepsilon r-\delta+1\leq r/2$ with
$0\leq\varepsilon,\delta\leq1$.\\

\smallskip
\noindent\textbf{(vii)} \textit{Estimate of $\| w_{3}^{n}\|_{L^{2}(  \Omega_{2})}$:}
Now we will estimate $\lambda_{n}^{2}\left\Vert w_{3}^{n}\right\Vert
_{L^{2}\left(  \Omega_{2}\right)  }^{2}$.

Inserting $i\lambda_{n}w_{3}^{n}-w_{4}^{n}=\mu_{n}^{-r}\widetilde{f}_{3}^{n}$
in (\ref{ast4}), we have%
$$
i\rho_{2}\lambda_{n}\left(  i\lambda_{n}w_{3}^{n}-\mu_{n}^{-r}\widetilde{f}%
_{3}^{n}\right)  -\beta_{2}\Delta w_{3}^{n}=\rho_{2} \mu_{n}^{-r}\widetilde{f}_{4}%
^{n}\text{.}%
$$
Then,%
\begin{equation}
-\rho_{2}\lambda_{n}^{2}w_{3}^{n}-\beta_{2}\Delta w_{3}^{n}=\rho_{2} \mu_{n}%
^{-r}\widetilde{f}_{4}^{n}+i\rho_{2}\lambda_{n}\mu_{n}^{-r}\widetilde{f}%
_{3}^{n}\text{.} \label{ast23}%
\end{equation}
We multiply (\ref{ast23}) by $q\cdot\nabla\overline{w_{3}^{n}}$ and integrate
over $\Omega_{2}$ in order to obtain that%
\begin{multline}
-\rho_{2}\lambda_{n}^{2}\int_{\Omega_{2}}w_{3}^{n}q\cdot\nabla\overline
{w_{3}^{n}}\,dx-\beta_{2}\int_{\Omega_{2}}\Delta w_{3}^{n}q\cdot\nabla
\overline{w_{3}^{n}}\,dx\\
=\mu_{n}^{-r}\int_{\Omega_{2}} \rho_{2} \left(  i
\lambda_{n}\widetilde{f}_{3}^{n}+\widetilde{f}_{4}^{n}\right)  q\cdot
\nabla\overline{w_{3}^{n}}\,dx\text{.} \label{ast24}%
\end{multline}
Using Rellich's identity,%
$$
\operatorname{Re}\int_{\Omega_{2}}\Delta w_{3}^{n}\left(  q\cdot
\nabla\overline{w_{3}^{n}}\right)\,dx = -\operatorname{Re}\int_{I}\partial_{\nu
}w_{3}^{n}\left(  q\cdot\nabla\overline{w_{3}^{n}}\right)  dS+\frac{1}{2}%
\int_{I}\left(  q\cdot\nu\right)  \left\vert \nabla w_{3}^{n}\right\vert ^{2} dS
$$
in the real part of (\ref{ast24}), we get%
\begin{multline*}
 -\rho_{2}\lambda_{n}^{2}\operatorname{Re}\int_{\Omega_{2}}w_{3}^{n} q\cdot\nabla\overline{w_{3}^{n}}\,dx \\
-\beta_{2}\left[  -\operatorname{Re}\int_{I}\partial_{\nu}w_{3}^{n}\left(  q\cdot\nabla\overline{w_{3}^{n}}\right)
dS + \frac{1}{2}\int_{I}\left(  q\cdot\nu\right)  \left\vert \nabla w_{3}^{n}\right\vert ^{2} dS\right] \\
  =\underbrace{\mu_{n}^{-r}\operatorname{Re}\int_{\Omega_{2}} \rho_{2}\left(
i\lambda_{n}\widetilde{f}_{3}^{n}+\widetilde{f}_{4}^{n}\right)
q\cdot\nabla\overline{w_{3}^{n}}\,dx}_{=:D_{n}}\text{,}%
\end{multline*}
or equivalently%
\begin{multline}
-\rho_{2}\lambda_{n}^{2}\operatorname{Re}\int_{\Omega_{2}}w_{3}^{n}%
q\cdot\nabla\overline{w_{3}^{n}}dx=D_{n}-\beta_{2}\operatorname{Re}\int%
_{I}\partial_{\nu}w_{3}^{n}\left(  q\cdot\nabla\overline{w_{3}^{n}}\right)
dS\\
+ \frac{1}{2}\beta_{2}\int_{I}\left(  q\cdot\nu\right)  \left\vert \nabla
w_{3}^{n}\right\vert ^{2} dS\text{.} \label{ast25a}%
\end{multline}
Making use of the identity $q\cdot\nabla\overline{w_{3}^{n}}%
=\operatorname{div}\left(  q\overline{w_{3}^{n}}\right)  -2\overline{w_{3}%
^{n}}$ and employing integration by parts, it holds%
\begin{equation}
\operatorname{Re}\int_{\Omega_{2}}w_{3}^{n}\left(  q\cdot\nabla\overline
{w_{3}^{n}}\right)  dx=-\left\Vert w_{3}^{n}\right\Vert _{L^{2}\left(
\Omega_{2}\right)  }^{2}-\frac{1}{2}\int_{I}\left(  q\cdot\nu\right)
\left\vert w_{3}^{n}\right\vert dS\text{.} \label{ast26}%
\end{equation}
By (\ref{ast25a}) and (\ref{ast26}), we get%
\begin{multline}\label{Eq_ast26b}
\rho_{2}\lambda_{n}^{2}\left\Vert w_{3}^{n}\right\Vert _{L^{2}\left(
\Omega_{2}\right)  }^{2}=D_{n}-\frac{1}{2}\rho_{2}\lambda_{n}^{2}\int%
_{I}\left(  q\cdot\nu\right)  \left\vert w_{3}^{n}\right\vert dS\\
-\beta_{2}\operatorname{Re}\int_{I}\partial_{\nu}w_{3}^{n}\left(  q\cdot
\nabla\overline{w_{3}^{n}}\right)  dS
+\frac{1}{2}\beta_{2}\int_{I}\left(
q\cdot\nu\right)  \left\vert \nabla w_{3}^{n}\right\vert ^{2} dS\text{.}%
\end{multline}
Note that%
$$
\left\vert \int_{I}\partial_{\nu}w_{3}^{n}\left(  q\cdot\nabla\overline
{w_{3}^{n}}\right)  dS\right\vert \leq\left\Vert \nu\cdot\nabla w_{3}%
^{n}\right\Vert _{L^{2}\left(  I\right)  }\left\Vert q\cdot\nabla
\overline{w_{3}^{n}}\right\Vert _{L^{2}\left(  I\right)  }\leq C\left\Vert
\nabla w_{3}^{n}\right\Vert _{L^{2}\left(  I\right)  ^{2}}^{2}%
$$
and%
$$
\left\vert \int_{I}\left(  q\cdot\nu\right)  \left\vert w_{3}^{n}\right\vert
dS\right\vert \leq C\left\Vert \nabla w_{3}^{n}\right\Vert _{L^{2}\left(
I\right)  ^{2}}^{2}\text{.}%
$$
Therefore%
\begin{align*}
\rho_{2}\lambda_{n}^{2}\left\Vert w_{3}^{n}\right\Vert _{L^{2}\left(
\Omega_{2}\right)  }^{2}  &  \leq C\bigg(  \mu_{n}^{-r}\big\Vert i\rho
_{2}\lambda_{n}\widetilde{f}_{3}^{n}+\widetilde{f}_{4}^{n}\big\Vert_{L^{2}\left(  \Omega_{2}\right)  }\underbrace{\left\Vert \nabla w_{3}^{n}\right\Vert _{L^{2}\left(  \Omega_{2}\right)^{2}}}_{\leq C} \\
&  \qquad\qquad\qquad\qquad\quad + \ \lambda_{n}^{2}\int_{I}\left\vert w_{3}^{n}\right\vert dS+\left\Vert \nabla w_{3}^{n}\right\Vert _{L^{2}\left(  I\right)  ^{2}}^{2}\bigg) \\
&  \leq C\left(  \mu_{n}^{-r}\mu_{n}\Vert \widetilde{f}^{n}\Vert
_{\mathscr{H}}+\lambda_{n}^{2}\left\Vert w_{1}^{n}\right\Vert _{L^{2}\left(
I\right)  }^{2}+\left\Vert \nabla w_{3}^{n}\right\Vert _{L^{2}\left(
I\right)  ^{2}}^{2}\right)  \text{,}%
\end{align*}
i.e.,%
\begin{equation}
\rho_{2}\lambda_{n}^{2}\left\Vert w_{3}^{n}\right\Vert _{L^{2}\left(
\Omega_{2}\right)  }^{2}\leq C\left(  \mu_{n}^{-r+1}\Vert \widetilde{f}^{n}\Vert _{\mathscr{H}}+\lambda_{n}^{2}\left\Vert w_{1}^{n}\right\Vert
_{L^{2}\left(  I\right)  }^{2}+\left\Vert \nabla w_{3}^{n}\right\Vert
_{L^{2}\left(  I\right)  ^{2}}^{2}\right)  \text{,} \label{ast27}%
\end{equation}
where $\mu_{n}^{-r+1}\leq C$ due to \color{black} $r\geq32/3$\color{black}.

Note that $w_{3}^{n}=w_{1}^{n}$ on $I$ implies $\partial_{\tau}w_{3}%
^{n}=\partial_{\tau}w_{1}^{n}$ on $I$, and thus%
\begin{align}
\left\Vert \nabla w_{3}^{n}\right\Vert _{L^{2}\left(  I\right)  ^{2}}^{2}  &
=\left\Vert \left(  \partial_{\tau}w_{3}^{n}\right)  \tau+\left(
\partial_{\nu}w_{3}^{n}\right)  \nu\right\Vert _{L^{2}\left(  I\right)  ^{2}%
}^{2}\nonumber\\
&  =\left\Vert \left(  \partial_{\tau}w_{1}^{n}\right)  \tau+\left(
\partial_{\nu}w_{3}^{n}\right)  \nu\right\Vert _{L^{2}\left(  I\right)  ^{2}%
}^{2}\nonumber\\
&  \leq C\left(  \left\Vert \partial_{\tau}w_{1}^{n}\right\Vert _{L^{2}\left(
I\right)  }^{2}+\left\Vert \partial_{\nu}w_{3}^{n}\right\Vert _{L^{2}\left(
I\right)  }^{2}\right)  \text{,} \label{ast27b}%
\end{align}
where%
\begin{align}
\left\Vert \partial_{\tau}w_{1}^{n}\right\Vert _{L^{2}\left(  I\right)  }^{2}
&  \leq C\left(  \left\Vert w_{1}^{n}\right\Vert _{H^{1}\left(  \Omega
_{1}\right)  }^{1/2}\left\Vert w_{1}^{n}\right\Vert _{H^{2}\left(  \Omega
_{1}\right)  }^{1/2}\right)  ^{2}\nonumber\\
&  \leq C\left\Vert w_{1}^{n}\right\Vert _{H^{1}\left(  \Omega_{1}\right)
}\nonumber\\
&  \leq C\left\Vert w_{1}^{n}\right\Vert _{L^{2}\left(  \Omega_{1}\right)
}^{1/2} \label{ast27c}%
\end{align}
due to Corollary \ref{coro-na} and (\ref{ast21a}).

As $w^{n}\in D(\mathscr{A})$, the following transmission condition in
the sense of the trace holds: 
$$
\beta_{2}\partial_{\nu}w_{3}^{n}=-\beta_{1}\mathscr{B}_{2}w_{1}^{n}%
+\alpha\kappa w_{5}^{n}\text{ \ \ on }I\text{.}%
$$
From this, from (\ref{ast27b}) and\ (\ref{ast27c}), we have%
\begin{equation}
\left\Vert \nabla w_{3}^{n}\right\Vert _{L^{2}\left(  I\right)  ^{2}}^{2}\leq
C\left(  \left\Vert w_{1}^{n}\right\Vert _{L^{2}\left(  \Omega_{1}\right)
}^{1/2}+\left\Vert \mathscr{B}_{2}w_{1}^{n}\right\Vert _{L^{2}\left(
I\right)  }^{2}+\left\Vert w_{5}^{n}\right\Vert _{L^{2}\left(  I\right)  }%
^{2}\right)  \text{.} \label{ast28}%
\end{equation}
By (3C.53) in \cite{LasieckaTriggiani2000Control}, we can write $B_{2}%
w_{1}^{n}=\partial_{\nu}\partial_{\tau}w_{1}^{n}$. Applying Corollary
\ref{coro-na} to $\Delta w_{1}^{n}$ and Lemma \ref{boundary} to
the partial derivative of third order of $w_{1}^{n}$ and using Sobolev's
interpolation inequality, we deduce that%
\begin{align*}
\left\Vert \mathscr{B}_{2}w_{1}^{n}\right\Vert _{L^{2}\left(  I\right)  }  &
=\left\Vert \partial_{\nu}\Delta w_{1}^{n}+\left(  1-\mu\right)
\partial_{\tau}\partial_{\nu}\partial_{\tau}w_{1}^{n}\right\Vert
_{L^{2}\left(  I\right)  }\\
&  \leq C\left(  \left\Vert \partial_{\nu}\Delta w_{1}^{n}\right\Vert
_{L^{2}\left(  I\right)  }+\left\Vert \partial_{\tau}\partial_{\nu}%
\partial_{\tau}w_{1}^{n}\right\Vert _{L^{2}\left(  I\right)  }\right) \\
&  \leq C\bigg(  \underbrace{\left\Vert \Delta w_{1}^{n}\right\Vert
_{H^{2}\left(  \Omega_{1}\right)  }^{1/2}}_{\leq\left\Vert w_{1}^{n}\right\Vert _{H^{4}\left(  \Omega_{1}\right)  }^{1/2}}\underbrace{\left\Vert \Delta w_{1}^{n}\right\Vert _{H^{1}\left(  \Omega_{1}\right)  }^{1/2}}_{\leq\left\Vert w_{1}^{n}\right\Vert _{H^{3}\left(
\Omega_{1}\right)  }^{1/2}}+\left\Vert w_{1}^{n}\right\Vert _{H^{4}\left(
\Omega_{1}\right)  }^{1/2}\left\Vert w_{1}^{n}\right\Vert _{H^{3}\left(
\Omega_{1}\right)  }^{1/2}\bigg) \\
&  \leq C\left\Vert w_{1}^{n}\right\Vert _{H^{4}\left(  \Omega_{1}\right)
}^{1/2}\left\Vert w_{1}^{n}\right\Vert _{H^{3}\left(  \Omega_{1}\right)
}^{1/2}\\
&  \leq C\left\Vert w_{1}^{n}\right\Vert _{H^{4}\left(  \Omega_{1}\right)
}^{1/2}\left(  \left\Vert w_{1}^{n}\right\Vert _{H^{4}\left(  \Omega
_{1}\right)  }^{3/4}\left\Vert w_{1}^{n}\right\Vert _{L^{2}\left(  \Omega
_{1}\right)  }^{1/4}\right)  ^{1/2}\\
&  \leq C\left\Vert w_{1}^{n}\right\Vert _{H^{4}\left(  \Omega_{1}\right)
}^{7/8}\left\Vert w_{1}^{n}\right\Vert _{L^{2}\left(  \Omega_{1}\right)
}^{1/8}\\
&  =C\mu_{n}^{7/8}\Big(  \underbrace{\mu_{n}^{-1}\left\Vert w_{1}^{n}\right\Vert _{H^{4}\left(  \Omega_{1}\right)  }}_{\leq C}\Big)^{7/8}\mu_{n}^{-\frac{\varepsilon r-\delta+1}{16}\cdot\frac{1}{8}}\\
& \qquad \qquad\qquad\qquad\qquad\quad \cdot \mu_{n}^{-1/8}\Big(  \mu_{n}^{\frac{\varepsilon r-\delta+1}{16}}\mu
_{n}\left\Vert w_{1}^{n}\right\Vert _{L^{2}\left(  \Omega_{1}\right)
}\Big)^{1/8}\\
&  \leq C\mu_{n}^{\frac{3}{4}-\frac{\varepsilon r-\delta+1}{16}\cdot\frac
{1}{8}}\left(  \mu_{n}^{\frac{\varepsilon r-\delta+1}{16}}\mu_{n}\left\Vert
w_{1}^{n}\right\Vert _{L^{2}\left(  \Omega_{1}\right)  }\right)
^{1/8}\text{.}%
\end{align*}
Then%
\begin{equation}
\left\Vert \mathscr{B}_{2}w_{1}^{n}\right\Vert _{L^{2}\left(  I\right)  }%
^{2}\leq C\mu_{n}^{\frac{3}{2}-\frac{\varepsilon r-\delta+1}{64}}\left(
\mu_{n}^{\frac{\varepsilon r-\delta+1}{16}}\mu_{n}\left\Vert w_{1}%
^{n}\right\Vert _{L^{2}\left(  \Omega_{1}\right)  }\right)  ^{1/4}%
\rightarrow0\text{,} \label{ast31}%
\end{equation}
due to (\ref{ast21}), whenever \color{black} $\frac{3}{2}-\frac{\varepsilon
r-\delta+1}{64}\leq0$ and $\varepsilon r-\delta+1\leq\frac{r}{2}$.
\color{black} Note that%
\begin{align*}
\frac{3}{2}-\frac{\varepsilon r-\delta+1}{64}  &  \leq0\text{ and }\varepsilon
r-\delta+1\leq\frac{r}{2}\\
&  \Longleftrightarrow96\leq\varepsilon r-\delta+1\leq\frac{r}{2}\\
&  \Longleftrightarrow190+2\delta\leq2\varepsilon r\leq r-2\left(
1-\delta\right)  \text{.}%
\end{align*}
To decrease the value of $r$ in the last inequality, $\delta$ should be zero
and $2\varepsilon$ should take its largest value such that $2\varepsilon r\leq
r-2$, this is, when $\delta=0$ and $2\varepsilon=1-\tfrac{2}{r}$. Therefore
$190\leq r-2$, i.e, \color{black} $r\geq192$. \color{black} In consequence
(\ref{ast31}) holds for \color{black} $r\geq192$. \color{black}

Thus, from the dissipation (\ref{estimate dissip}), from \eqref{ast21d}, \eqref{ast21c}, \eqref{ast27}, \eqref{ast28}, and \eqref{ast31}, we obtain%
\begin{equation}
\lambda_{n}^{2}\left\Vert w_{3}^{n}\right\Vert _{L^{2}\left(  \Omega
_{2}\right)  }^{2}\rightarrow0\text{.}\label{ast34}%
\end{equation}

\smallskip
\noindent\textbf{(viii)} \textit{Estimate of $\| w_{4}^{n}\|_{L^{2}(  \Omega_{2})}$:} From the equivalence of the norms in $\HH$ and $\mathscr{X}$, and \eqref{ec resolvent2}, it holds
\begin{equation*}
      i\lambda_{n} w_{3}^{n}- w_{4}^{n} = \mu_{n}^{-r}\widetilde{f}_{3}^{n} \to 0 \quad \text{in}\quad L^2(\Omega_2).
\end{equation*}
From this and \eqref{ast34} we obtain
\begin{equation*}
    \left\Vert w_{4}^{n}\right\Vert _{L^{2}\left(  \Omega
_{2}\right)  }\rightarrow0\text{.}\label{ast34b}%
\end{equation*}

\smallskip
\noindent\textbf{(ix)} \textit{End of the proof:}
Now we will show that
$$\beta_{1}\left\Vert w_{1}^{n}\right\Vert _{H_{\Gamma
}^{2}\left(  \Omega_{2}\right)  }^{2}+\beta_{2}\left\Vert \nabla w_{3}%
^{n}\right\Vert _{L^{2}\left(  \Omega_{2}\right)  }^{2}\rightarrow0,$$ when
$r\geq192$. For this, it is enough to prove that the following terms in
\eqref{ast10}: $(  w_{5}^{n},\partial_{\nu}w_{1}^{n})_{L^{2}\left(  I\right)  }$, $\mu_{n}^{-r}(  \widetilde{f}_{2}^{n},w_{1}^{n})_{L^{2}\left(  \Omega_{1}\right)  }$, $\lambda_{n}\mu
_{n}^{-r}(  \widetilde{f}_{1}^{n},w_{1}^{n})_{L^{2}\left(
\Omega_{1}\right)  }$, $(  \nabla w_{5}^{n},\nabla w_{1}^{n})_{L^{2}\left(  \Omega_{1}\right)  ^{2}}$, $\mu_{n}^{-r}(  \widetilde{f}_{4}^{n},w_{3}^{n})_{L^{2}\left(  \Omega_{2}\right)  }$, and
$\lambda_{n}\mu_{n}^{-r}(  \widetilde{f}_{3}^{n},w_{3}^{n})_{L^{2}\left(  \Omega_{2}\right)  }$, go to zero when $n\rightarrow\infty$, but
that follows from $\left\Vert w^{n}\right\Vert _{\mathscr{H}}=1$, $\Vert
\widetilde{f}^{n}\Vert _{\mathscr{H}}\to0$, and the dissipation.
In fact:%

\begin{align*}
\big\vert \left(  w_{5}^{n},\partial_{\nu}w_{1}^{n}\right)_{L^{2}\left(
I\right)  }\big\vert & \leq \left\Vert w_{5}^{n}\right\Vert _{L^{2}\left(
I\right)  }\left\Vert \partial_{\nu}w_{1}^{n}\right\Vert _{L^{2}\left(
I\right)  }\\
& \leq C\left\Vert w_{5}^{n}\right\Vert _{L^{2}\left(  I\right)
}\underbrace{\left\Vert w_{1}^{n}\right\Vert _{H^{2}\left(  \Omega_{1}\right)
}}_{\leq C}\rightarrow0\text{,}%
\end{align*}
$$
\big\vert \mu_{n}^{-r}(  \widetilde{f}_{2}^{n},w_{1}^{n})
_{L^{2}\left(  \Omega_{1}\right)  }\big\vert \leq C\mu_{n}^{-r}\Vert
\widetilde{f}_{2}^{n}\Vert _{L^{2}\left(  \Omega_{1}\right)
}\underbrace{\left\Vert w_{1}^{n}\right\Vert _{H^{2}\left(  \Omega_{1}\right)
}}_{\leq C}\rightarrow0\text{,}%
$$%
$$
\big\vert \lambda_{n}\mu_{n}^{-r}(  \widetilde{f}_{1}^{n},w_{1}^{n})_{L^{2}\left(  \Omega_{1}\right)  }\big\vert \leq C\mu
_{n}^{1-r}\Vert \widetilde{f}_{1}^{n}\Vert _{L^{2}\left(
\Omega_{1}\right)  }\underbrace{\left\Vert w_{1}^{n}\right\Vert _{H^{2}\left(
\Omega_{1}\right)  }}_{\leq C}\rightarrow0\text{,}%
$$%
$$
\big\vert (  \nabla w_{5}^{n},\nabla w_{1}^{n})_{L^{2}\left(
\Omega_{1}\right)^{2}}\big\vert \leq C\left\Vert \nabla w_{5}^{n}\right\Vert _{L^{2}\left(  \Omega_{1}\right)  ^{2}}\underbrace{\left\Vert
w_{1}^{n}\right\Vert _{H^{2}\left(  \Omega_{1}\right)  }}_{\leq c}%
\rightarrow0\text{,}%
$$%
$$
\big\vert \mu_{n}^{-r}(  \widetilde{f}_{4}^{n},w_{3}^{n})
_{L^{2}\left(  \Omega_{2}\right)  }\big\vert \leq\mu_{n}^{-r}\Vert
\widetilde{f}_{4}^{n}\Vert _{L^{2}\left(  \Omega_{2}\right)
}\underbrace{\left\Vert w_{3}^{n}\right\Vert _{L^{2}\left(  \Omega_{1}\right)
}}_{\leq c}\rightarrow0\text{,}%
$$%
$$
\big\vert \lambda_{n}\mu_{n}^{-r}(  \widetilde{f}_{3}^{n},w_{3}^{n})  _{L^{2}\left(  \Omega_{2}\right)  }\big\vert \leq C\mu
_{n}^{1-r}\Vert \widetilde{f}^{n}\Vert _{\mathscr{H}}%
\underbrace{\left\Vert w^{n}\right\Vert _{\mathscr{H}}}_{=1}\rightarrow
0\text{.}%
$$
Therefore%
\begin{equation}
\beta_{1}\left\Vert w_{1}^{n}\right\Vert _{H_{\Gamma}^{2}\left(  \Omega
_{2}\right)  }^{2}+\beta_{2}\left\Vert \nabla w_{3}^{n}\right\Vert
_{L^{2}\left(  \Omega_{2}\right)  }^{2}\rightarrow0\text{,} \label{ast35}%
\end{equation}
also due to  \eqref{ast10}, \eqref{ast21} and \eqref{ast34}. In consequence,
$\left\Vert w^{n}\right\Vert _{\mathscr{H}}\rightarrow0$, which is a
contradiction with $\left\Vert w^{n}\right\Vert _{\mathscr{H}}=1$.

To complete the proof we still have to show that $i\mathbb{R}\subset \rho(\AA)$. This is done in Lemma \ref{Prop_resolvent}.

\begin{remark}
With this proof of Theorem \ref{Th poly without GC}, the decay rate is 
at least $1/192$ regardless of the presence or absence of the geometrical condition
\eqref{GC} on $I$. We expect this rate to be not optimal. However, it seems to be hard to obtain optimal polynomial rates except in special situations, where one can determine the eigenvalues of the operator $\mathscr A$ explicitly.
\end{remark}

\begin{lemma}\label{Prop_resolvent}
Let $m= 0$. Then the imaginary axis is contained in the resolvent set of $\AA$, i.e., $i\R\subset\rh(\AA)$.
\end{lemma}


\begin{proof}
 As $0\in\rh(\AA)$ (see Lemma \ref{Prop_IR}) and  $\rh(\AA)$ is open in $\C$, then there is $R>0$ such that $B(0, R)\subset\rh(\AA)$. For real $\de$ such that $0<\delta< R$, we have $i[-\de, \de]\subset\rh(\AA)$.
Thus, $\RR\coloneqq \{\la>0 : i[-\la, \la]\subset\rh(\AA)\}$ is not empty. Let $\la^*\coloneqq \sup\RR$. If $\la^*=\infty$, we have nothing to prove. Let us suppose $\la^*<\infty$. If $\la^*\in\RR$, then $i[-\la^*, \la^*]$ is contained in $\rh(\AA)$ and so it is possible to find $M>0$ such that $i[-\la^*-M, \la^*+M]\subset\rh(\AA)$. The above indicates that $\la^*+M\in\RR$ which contradicts the assumption that $\la^*$ is the supremum of $\RR$. Hence, $\la^*\not\in\RR$. Then, there exists $\left(\la_n\right)_{n\in\N}\subset\RR$ such that $\lim_{n\rightarrow\infty}\la_n=\la^*$ and $\lim_{n\rightarrow\infty}\left\|(i\la_n\Ii-\AA)^{-1}\right\|_{\mathscr{L}(\HH)}=\infty$. Hence, there exists $(\tilde{f}_n)_{n\in\N}\subset\HH$ with
$$
\|\tilde{f}_n\|_{\HH}=1 \quad\textup{and}\quad \lim_{n\rightarrow\infty}\|(i\la_n\Ii-\AA)^{-1}\tilde{f}_n\|_{\HH}=\infty.
$$
Here putting $\tilde{w}_n\coloneqq (i\la_n\Ii-\AA)^{-1}\tilde{f}_n$, $w_n\coloneqq \tilde{w}_n/\|\tilde{w}_n\|_{\HH}$ and $f_n\coloneqq \tilde{f}_n/\|\tilde{w}_n\|_{\HH}$, we get that
\begin{equation*}\label{E1}
\left(i\la_n\Ii-\AA\right)w_n=f_n
\end{equation*}
with
\begin{equation*}\label{E2}
\left\|w_n\right\|_{\HH}=1
\end{equation*}
and
\begin{equation*}\label{E3}
\left\|(i\la_n\Ii-\AA)w_n\right\|_{\HH}\xrightarrow[n\rightarrow\infty]{}0.
\end{equation*}
In the same way as in  the proof above, we obtain  $\left\|w_n\right\|_{\HH}\to 0$.
\end{proof}


%

\begin{appendix}
\section{Auxiliary results}\label{Annexes}

In this appendix, we collect some results that are useful for the proofs of the main results.

\begin{lemma}[{\cite[Theorem 1.4.4]{liu1999semigroups}}]\label{boundary}
Let $U$ be a bounded domain in $\R^n$ with $C^1$ boundary $\pa U$. For any function $u\in H^1(U)$, the following estimate holds:
\begin{equation}\label{EQU3}
\Vert u\Vert_{L^2(\partial U)} \leq C \Vert u\Vert_{H^1(U)}^{1/2}\Vert u\Vert_{L^2(U)}^{1/2}
\end{equation}
with $C$ being a positive constant independent of $u$.
\end{lemma}

\begin{corollary}\label{coro-na}
For any function $u\in H^2(U)$,  the following estimate holds:
\begin{equation}\label{Ineq_02}
\left\|\pa_\nu u\right\|_{L^2(\pa U)}\leq\widetilde{C} \left\|u\right\|_{H^2(U)}^{1/2}\left\|\na u\right\|_{L^2(U)^n}^{1/2}
\end{equation}
with $\widetilde{C}>0$ independent of $u$. Estimate \eqref{Ineq_02} is also true if we change $\nu$ by $\tau$.
\end{corollary}
\begin{proof}
An application of Lemma \ref{boundary} to the first partial derivatives of $u$ leads to \eqref{Ineq_02}.
\end{proof}

\begin{lemma}\label{lemmaAN}
Let $a, b$ and $c$ be positive real constants with $ab > c$ and let
\[
p(\lambda) \ceqq \lambda^3 + a\lambda^2 + b\lambda + c,\quad \lambda \in \mathbb{C}.
\]
Then,
\begin{itemize}
    \item[$i)$] $p(\lambda)$ has a negative real root.
    \item[$ii)$] $p(\lambda)$ has no pure imaginary roots, that is, $p(iy) \neq 0$ for all $y \in \mathbb{R}$.
    \item[$iii)$] If $\gamma = x + i y$ is a root of $p(\lambda)$ with $x, y \in \mathbb{R}\smallsetminus\{0\}$, then $x$ is a solution of the polynomial equation
    \begin{equation} \label{real poly ec}
     8x^3 + 8a x^2 + 2(a^2 + b)x + ab - c = 0.
    \end{equation}
    Therefore, $x < 0$.
   \item[$iv)$] $p(\lambda) \neq 0$ for all $\lambda \in \overline{\mathbb{C}_+}$.
\end{itemize}
\end{lemma}
\begin{proof}
$i)$ is an immediate consequence of the positivity of the constants $a, b$ and $c$, and $ii)$ is a consequence of $c \neq ab$. Now, we prove $iii)$. Let $\gamma = x + i y$ be a root of $p(\lambda)$ with $x, y\in \R\smallsetminus \{0\}$. Solving $y^2$ from $\Im(p(\lambda)) = 0$ and substituting in $\Re(p(\lambda)) = 0$, we get (\ref{real poly ec}). $iv)$ is a consequence of $i)$-$iii)$.

\end{proof}

For example, if the Greek letters below represent positive constants, then $a\ceqq\dfrac{\beta}{\rho_0}$, $b\ceqq\dfrac{\alpha^2 +  \beta_1\rho_0}{\rho_0 \rho_1}$ and $c\ceqq\dfrac{\beta\beta_1}{\rho_0 \rho_1}$ satisfy the assumptions in the previous lemma, due to $\alpha^2 + \beta_1\rho_0  > \beta_1\rho_0 $.

\begin{lemma}[Aubin-Lions] \label{A.4} Let $X_0$, $X$, $X_1$ be Hilbert spaces with $X_0\overset{c}{\hookrightarrow} X \hookrightarrow X_1$, and $T>0$. Then it holds
$$ \big\{u\in L^2((0,T),X_0)\,:\, u_t\in L^2((0,T),X_1)\big\} \overset{c}{\hookrightarrow} L^2((0,T), X). $$
\end{lemma}

\begin{proof}
See \cite[Proposition III.1.3]{Showalter97}.
\end{proof}

\end{appendix}



\end{document}